\documentclass{siamltex}
\usepackage{amsfonts,amsmath,amssymb}
\usepackage{mathrsfs,mathtools}
\usepackage{enumerate}
\usepackage{hyperref}
\usepackage{esint}
\usepackage{graphicx}
\usepackage{bm}
\usepackage{commath}
\usepackage{esint}
\DeclareMathAlphabet{\mathpzc}{OT1}{pzc}{m}{it}

\usepackage{color}
\newcommand{\HA}[1]{\textcolor{red}{#1}}

\renewcommand{\grad}{\nabla}

\newtheorem{remark}[theorem]{Remark}

\newcommand{\bbN}{\mathbb{N}}

\newcommand{\bbR}{\mathbb{R}}

\newcommand{\calA}{\mathcal{A}}
\newcommand{\calB}{\mathcal{B}}
\newcommand{\calC}{\mathcal{C}}
\newcommand{\calD}{\mathcal{D}}
\newcommand{\calE}{\mathcal{E}}
\newcommand{\calF}{\mathcal{F}}

\newcommand{\calI}{\mathcal{I}}

\newcommand{\calK}{\mathcal{K}}
\newcommand{\calL}{\mathcal{L}}

\newcommand{\calP}{\mathcal{P}}

\newcommand{\calR}{\mathcal{R}}
\newcommand{\calS}{\mathcal{S}}
\newcommand{\calT}{\mathcal{T}}
\newcommand{\calU}{\mathcal{U}}
\newcommand{\calV}{\mathcal{V}}

\newcommand{\mat}[1]{\bm{{#1}}}

\newcommand{\matI}{\mat I}

\newcommand{\matM}{\mat M}

\newcommand{\matP}{\mat P}

\newcommand{\matT}{\mat T}

\newcommand{\pe}{\textup{\textsf{p}}}
\newcommand{\qe}{\textup{\textsf{q}}}
\newcommand{\ue}{\textup{\textsf{u}}}
\newcommand{\ve}{\textup{\textsf{v}}}
\newcommand{\we}{\textup{\textsf{w}}}
\newcommand{\xe}{\textup{\textsf{x}}}
\newcommand{\ye}{\textup{\textsf{y}}}

\newcommand{\Be}{\textup{\textsf{B}}}
\newcommand{\De}{\textup{\textsf{D}}}

\newcommand{\Qe}{\textup{\textsf{Q}}}
\newcommand{\Ue}{\textup{\textsf{U}}}

\newcommand{\bdy}{\partial}

\newcommand{\ceil}[1]{\left\lceil{#1}\right\rceil}
\newcommand{\cembed}[2]{#1\hookrightarrow #2}

\newcommand{\clos}[1]{{\overline{#1}}}
\DeclareMathOperator{\codim}{codim}

\newcommand{\composed}{\circ}

\DeclareMathOperator{\curl}{curl}

\newcommand{\CCinf}[1]{C_{c}^\infty\of{#1}}

\DeclareMathOperator{\distText}{dist}
\newcommand{\dist}{\distText}
\DeclareMathOperator{\divgText}{div}

\newcommand{\divg}{\divgText}

\newcommand{\dual}[1]{{#1}^*}

\newcommand{\eclass}[1]{[#1]}

\providecommand{\grad}{\nabla}

\newcommand{\gradS}[1]{\vec\varepsilon\of{\vec{#1}}}

\newcommand{\halfspace}{{\bbR^n_{-}}}

\DeclareMathOperator{\ind}{ind}

\newcommand{\Linf}[1]{{L^\infty\of{#1}}}

\newcommand{\Lr}[1]{L^{r}\of{#1}}
\newcommand{\LrZ}[1]{L^r_{0}\of{#1}}
\newcommand{\Lrloc}[1]{L^{r}_{\loc}\of{#1}}
\newcommand{\Lrp}[1]{L^{r'}\of{#1}}

\newcommand{\LrpZ}[1]{L^{r'}_{0}\of{#1}}

\newcommand{\lapl}{\Delta}

\newcommand{\loc}{\textnormal{loc}}

\newcommand{\normL}[3]{\norm{#1}_{L^{#2}\of{#3}}}

\newcommand{\normLi}[2]{\norm{#1}_\Linf{#2}}

\newcommand{\normLr}[2]{\norm{#1}_{\Lr{#2}}}

\newcommand{\normS}[4]{\norm{#1}_{\sob{#2}{#3}{#4}}}

\newcommand{\normSH}[4]{\norm{#1}_{\sobH{#2}{#3}{#4}}}

\newcommand{\normSZ}[4]{\abs{#1}_{\sob{#2}{#3}{#4}}}

\newcommand{\of}[1]{(#1)}
\newcommand{\pair}[1]{\left\langle #1 \right\rangle}

\renewcommand{\restriction}{\big\vert}

\newcommand{\seq}[1]{\cbr{#1}}

\newcommand{\sob}[3]{W^{#1}_{#2}\of{#3}}
\newcommand{\sobH}[3]{G^{#1}_{#2}\of{#3}}
\newcommand{\sobHZ}[3]{\mathring G^{#1}_{#2}\of{#3}}

\newcommand{\sobZ}[3]{{\mathring{W}^{#1}_{#2}\of{#3}}}

\newcommand{\transpose}{^\top}

\renewcommand{\vec}[1]{\bm{{#1}}}
\newcommand{\vecH}[1]{\bm{{\hat{#1}}}}

\newcommand{\vnu}{{\vec{\nu}}}

\renewcommand{\divg}[1]{\divgText #1}

\begin{document}

\title{The Stokes problem with Navier slip boundary condition: 
  Minimal fractional Sobolev Regularity of the domain\thanks{HA has been partially supported by NSF grants DMS-1109325 and DMS-1521590. RHN has been partially supported by NSF grants DMS-1109325 and DMS-1411808. PS has been partially supported by NSF grant DMS-1109325.}}

\author{Harbir Antil\thanks{Department of Mathematical Sciences, George Mason University, Fairfax, VA 22030, USA ({\tt hantil@gmu.edu})}
\and 
Ricardo H. Nochetto\thanks{Department of Mathematics and Institute for Physical Science and Technology, University of Maryland
College Park, MD 20742, USA ({\tt rhn@math.umd.edu})} 
\and 
Patrick Sodr{\'e}\thanks{Department of Mathematics, University of Maryland
College Park, MD 20742, USA ({\tt sodre@math.umd.edu})}.}

\pagestyle{myheadings}
\thispagestyle{plain}
\markboth{H.~Antil, R.H.~Nochetto, P.~Sodr{\'e}}{Stokes with Navier slip conditions: $W^{2-1/s}_{s}$
  Domains}


\maketitle

\begin{abstract}
We prove well-posedness in reflexive Sobolev spaces of weak
  solutions to the stationary Stokes problem with Navier slip boundary
  condition over bounded domains $\Omega$ of $\bbR^n$ of class
  $W^{2-1/s}_s$, $s>n$. Since such domains are of class $C^{1,1-n/s}$,
  our result improves upon the recent one by Amrouche-Seloula, who
  assume $\Omega$ to be of class $C^{1,1}$. We deal with the slip
  boundary condition directly via a new localization technique, which
  features domain, space and operator decompositions.
  To flatten the boundary of $\Omega$ locally, we construct a novel
  $W^2_s$ diffeomorphism for $\Omega$ of class $W^{2-1/s}_s$. The
  fractional regularity gain, from $2-1/s$ to $2$, guarantees that the
  Piola transform is of class $W^1_s$. This allows us to transform
  $W^1_r$ vector fields without changing their regularity, provided
  $r\le s$, and preserve the unit normal which is H\"older. It is in
  this sense that the boundary regularity $W^{2-1/s}_s$ seems to be minimal.
\end{abstract}



\section{Introduction}
A bounded connected domain $\Omega$ in $\bbR^n$ ($n \geq 2$) is said to be of fractional Sobolev class $W^{2-1/s}_{s}$, ($n < s < \infty$) whenever its boundary $\bdy\Omega$ is locally the graph of a function $\omega$ in $\sob{2-1/s}{s,\loc}{\bbR^{n-1}}$. 
We refer to \cite{MDelfour_JZolesio_1998a,MDelfour_JZolesio_2011a}
for an equivalent definition via the signed distance function.
Our primary goal is to establish well posenedness (in the sense of Hadamard)
of the following Stokes problem
\begin{subequations}\label{eq:stokes_full}
	\begin{equation}\label{eq:stokes}
		-\divg{\vec\sigma\del{\vec u, p}} = \vec f, \quad \divg{\vec u} = g \quad \mbox{in } \Omega,	\end{equation}
together with the Navier slip boundary condition, 
	\begin{equation}\label{eq:stokes_bc}
	\vec u\cdot\vec{\nu} = \phi,\quad \beta \matT \vec u + \matT\transpose\vec\sigma\del{\vec u, p}\vec\nu = \vec\psi \quad \mbox{on } \bdy\Omega,
	\end{equation}
\end{subequations}
in the reflexive Sobolev spaces $W^1_r(\Omega) \times L^r(\Omega)$ with
$s' = s/(s-1) \le r \le s$.
Here $\vec\sigma = \eta\vec\varepsilon\of{\vec u} - \vec Ip$ is the stress tensor, 
$\eta$ is a constant viscosity parameter (Newtonian fluid), $\gradS u
= (\grad \vec u + \grad \vec u\transpose)$ is the strain tensor (or
symmetric gradient), $\vec\nu$ is the exterior unit normal to
$\bdy\Omega$, $\beta\of{\vec x} \geq 0$ is the friction coefficient,
and $\vec T = \vec I - \vnu\otimes\vnu$ is the projection operator
onto the  tangent plane of $\bdy\Omega$. Notice that when $\phi = 0$,
$\vec \psi = 0$, and $\beta = 0$ in \eqref{eq:stokes_bc} then the
fluid slips along the boundary. The well-known no-slip condition $\vec
u = \vec 0$ can be viewed as the limit of \eqref{eq:stokes_bc} when
$\phi = 0$, $\vec \psi = \vec 0$, and $\beta \rightarrow \infty$.
The boundary condition \eqref{eq:stokes_bc} is appropriate
in dealing with free boundary problems;
we refer to \cite{BJJin_MPadula_2004a, BJJin_2005a, VASolonnikov_1995a,
  MPadula_VASolonnikov_2010, HBae_2011a, JANitsche_1986,
  PFAntonietti_PFFadel_MVerani_2011a, HAntil_RHNochetto_PSodre_2014b} and related references \cite{CAmrouche_NEHSeloula_2011b,MPadula_VASolonnikov_2010}.

In contrast to \eqref{eq:stokes_bc},
the no-slip condition $\vec u = \vec 0$ has received a great deal of attention in
the literature. It turns out that most of what is valid for the
Poisson equation $-\Delta u = \divg \vec f$
extends to the Stokes equation with no-slip condition.
The a priori bound in $W^1_r(\Omega)$
\begin{equation}\label{weak-laplace}
\|u\|_{W^1_r(\Omega)} \le C \|\vec f\|_{L^r(\Omega)}
\end{equation}
is valid for the Laplacian
provided $\partial\Omega$ is Lipschitz and $s'\le r \le s$ for
$s>n$; see \cite[Theorem 1.1]{MR1331981} for Dirichlet condition
and \cite[Theorem~1.6]{DZZAnger_2000} for Neumann condition.
Moreover, the range of
$r$ becomes $1<r<\infty$ provided $\partial\Omega$ is $C^1$; see
\cite[Theorem 1.1]{MR1331981} for the Dirichlet problem.
A similar bound with $1 < r < \infty$ is valid for the Stokes system
with Dirichlet condition provided $\Omega$ is Lipschitz with a small
constant, and in particular for $C^1$ domains \cite{GPGaldi_CGSimader_HSohr_1994}.
This extends to Besov spaces on Lipschitz domains for the
Stokes system with  Dirichlet
and Neumann conditions  \cite{MMitrea_MWright_2012}; we refer to
\cite{MR975122, MR975121} for earlier contributions.
It is thus natural to wonder whether such estimates
would extend to the Navier boundary condition \eqref{eq:stokes_bc}
with a domain regularity weaker than $C^{1,1}$. We first mention some
relevant literature and next argue that the domain regularity
$W^{2-1/s}_s$ seems to be minimal.

It is well-known that in case $\phi = 0, \vec \psi = \vec 0$, and $\beta = 0$ one can write \eqref{eq:stokes_bc} equivalently as 
\begin{equation}\label{eq:curlf}
 \vec u\cdot\vec{\nu} = 0, \quad \vec\nu \times \curl \vec u = \vec 0 \quad \mbox{on } \bdy\Omega . 
\end{equation}
However, such an equivalence is known to hold only when the boundary
is of class $C^2$ \cite[Section~2]{MMitrea_SMonniaux_2009a}.
Starting from \eqref{eq:curlf}, Amrouche--Seloula
\cite{CAmrouche_NEHSeloula_2011b} obtained recently well posedness of
weak, ultra-weak, and strong solutions to \eqref{eq:stokes} in
reflexive Sobolev spaces for $\Omega$ of class $C^{1,1}$. To our
knowledge, this is the best result in the literature.
Our approach deals directly with \eqref{eq:stokes_bc} which makes it
amenable to free boundary problems. We elaborate below on this matter.

The pioneering work on the Stokes problem \eqref{eq:stokes_full} in
H\"older spaces is due to Solonnikov-\v S\v cadilov
\cite{VASolonnikov_VEScadilov_1973}. This was extended by Beir\~ao da
Veiga \cite{HBeiraodaVeiga_2004}, who showed existence of weak and
strong solutions to a generalized Stokes system with $C^{1,1}$-domains
in Hilbert space setting.
We also refer to Mitrea-Monniaux \cite{MMitrea_SMonniaux_2009a} who
proved existence of mild solutions in Lipschitz domains for
the time-dependent Navier-Stokes (NS) equations with boundary condition
\eqref{eq:curlf}.
A survey for the stationary and the time-dependent Stokes and
NS equations with slip boundary condition is given by Berselli \cite{LCBerselli_2010b}.
Finally, well posedness for several variants of the time-dependent
NS equations are shown by M\'{a}lek and collaborators for
$C^{1,1}$ domains
\cite{MBulicek_EFeireisl_JMalek_2009a,MBulicek_PGwiazda_JMalek_ASGwiazda_2012a,%
MBulicek_JMalek_KRRajagopal_2007a,MBulicek_JMalek_KRRajagopal_2009a}.

We now give three reasons why we find the
regularity $W^{2-1/s}_s$ of $\partial\Omega$ (nearly) {\it minimal} for
\eqref{eq:stokes_full}, and stress that they all hinge on the critical role played
by the unit normal $\vec \nu$.
We start with the boundary condition $\vec u \cdot \vec \nu = \phi$ in
\eqref{eq:stokes_bc}. Given the scalar function $\phi \in W^{1-1/r}_r(\partial\Omega)$,
in \S \ref{s:lifting} we construct a
vector-valued extension $\vec \varphi \in W^1_r(\Omega)$ such that
$\vec \varphi = \phi\vec \nu$ on $\partial\Omega$ in the sense of
traces along with $\|\vec \varphi \|_{W^1_r(\Omega)} \le C
\|\phi\|_{W^{1-1/r}_r(\partial\Omega)}$.
One possibility is to solve
the following auxiliary problem subject to the compatibility condition
$\int_\Omega g + \int_{\partial\Omega} \phi = 0$
\begin{equation}\label{aux-prob}
  -\Delta \xi = g \quad \textrm{in } \Omega
  \qquad
  \partial_\nu \xi = \phi \quad \textrm{on } \partial\Omega,
\end{equation}
and set $\vec\varphi:=\nabla\xi$. The requisite regularity $\xi\in W^2_r(\Omega)$,
whence $\vec\varphi \in W^1_r(\Omega)$, fails in general for
$\partial\Omega$ Lipschitz or even $C^1$ \cite{DJerison_CEKenig_1989a}.
If $\partial\Omega$ is of class $W^{2-1/s}_s$, instead, then we can
extend each component of $\vec\nu \in W^{1-1/s}_s$ and thus $\vec\nu$
to a function in $W^1_s(\Omega)$ (still denoted $\vec \nu$). If we
also denote $\phi\in W^1_r(\Omega)$ an extension of $\phi$, then a
simple calculation shows that $r\le s$ and $s>n$ yields
$\phi\vec \nu \in W^1_r(\Omega)$ and 
\begin{equation}\label{apriori}
\|\vec \varphi \|_{W^1_r(\Omega)} \le C
\|\phi\|_{W^{1-1/r}_r(\partial\Omega)} \|\vec \nu\|_{W^{1-1/s}_s(\partial\Omega)}.
\end{equation}
Setting $\vec v := \vec u - \vec \varphi$ we get a problem for $\vec v$
similar to \eqref{eq:stokes_full} with modified data with the same
regularity as $\vec f, g$ and $\vec\psi$ but $\vec v\cdot\vec \nu = 0$. We
study this problem in \S\S \ref{s:hilbert_case}--\ref{s:sobolev_case}.

Our localization technique is
the second instance for $\partial\Omega$ to be of class $W^{2-1/s}_s$.
In \S \ref{s:sobolev_piola} we construct a local $W^2_s$ diffeomorphism
$\vec\Psi$, such that $\vec\Psi^{-1}$
flattens $\partial\Omega$, by suitably extending the function $\omega\in W^{2-1/s}_s$
describing $\partial\Omega$ locally. We exploit this small gain of
regularity, from $2-1/s$ to $2$,
to define the {\it Piola transform} inverse
$\vec P^{-1}:=\det(\nabla\vec\Psi^{-1})\nabla\vec\Psi^{-1} \in W^1_s(\Omega)$,
which maps vector fields $\vec v\in W^1_r(\Omega)$ into vector fields
$(\vec P^{-1} \vec v) \circ \vec \Psi \in W^1_r(\bbR^n_-)$ with the same
divergence and normal trace in $\mathbb{R}^n_+$. This is instrumental to reduce
\eqref{eq:stokes_full} locally to a Stokes problem in $\mathbb{R}^n_-$ with variable
coefficients and Navier condition
\eqref{eq:stokes_bc} on $\{x_n=0\}$, and next make use of reflection
arguments. We develop the localization framework in \S \ref{s:sobolev_case}.

Our primary interest in studying \eqref{eq:stokes_full} with minimal domain
regularity is the Stokes problem defined in $\Omega$
\[
\Omega := \{(x,y)\in\mathbb{R}^2: 0 < x < 1, 0 < y < 1 + \gamma(x)\},
\quad
\Gamma := \{(x,1+\gamma(x)): 0<x<1\}
\]
with free boundary $\Gamma$ given by the graph of $1+\gamma$ and
$\gamma(0)=\gamma(1)=0$. The
free boundary condition on $\Gamma$ corresponds to surface tension
effects, is overdetermined, and reads
\begin{equation}\label{surf-tension}
  \vec u\cdot\vec \nu = 0,
  \quad
  \vec\sigma(\vec u,p)\vec \nu = \chi H[\gamma] \vec\nu,
\end{equation}
with $H[\gamma] = - \dif_{\, x}\big(\frac{\dif_{\, x}\gamma}{Q[\gamma]}\big)$
being the curvature of
$\Gamma$, $Q[\gamma] := \sqrt{1+|\dif_{\, x}\gamma|^2}$,
and $\chi>0$ a surface tension coefficient, whereas
a Dirichlet condition is imposed on the rest of $\partial\Omega$.
We realize that \eqref{surf-tension} includes \eqref{eq:stokes_bc}
besides the equation for the balance of forces
$\vec\nu^T\vec\sigma(\vec u,p)\vec \nu = \chi H[\gamma]$ which
determines the location of $\Gamma$.
To formulate this problem variationally, we multiply the momentum equation
in \eqref{eq:stokes} by $\vec v$, integrate by parts
\begin{equation*}
\int_\Omega \vec\sigma(\vec u,p) : \nabla\vec v - \vec f \cdot \vec v
= \int_\Gamma \vec v \cdot \vec\sigma(\vec u,p) \vec \nu
= \int_\Gamma H[\gamma] \vec v \cdot \vec \nu
= \int_0^1 \frac{\dif_{\, x}\gamma \dif_{\, x} v}{Q[\gamma]} 
\end{equation*}
and use $\gamma(0)=\gamma(1)=0$, where $v = Q[\gamma] \, \vec v \cdot \vec \nu$.
We emphasize again the critical role
that $\vec \nu$ plays: applying the Implicit Function Theorem enables us to prove that
$\gamma \in W^{2-1/s}_s(0,1)\subset W^1_\infty(0,1)$, $s>2$, whence
$\vec v\in W^1_{s'}(\Omega)$ implies $Q[\gamma] \, \vec v \cdot
\vec\nu \in W^{1-1/{s'}}_{s'}(\Gamma)$ as already alluded to in
\eqref{apriori}. This would not be possible with mere Lipschitz
regularity $\gamma\in W^1_\infty(0,1)$. We refer to
\cite{HAntil_RHNochetto_PSodre_2014b} for full details.

For the moment, we will make two simplifications in
\eqref{eq:stokes_bc}. The first one is to treat the frictionless problem,
i.e. $\beta = 0$; we will return to $\beta \neq 0$ in
\secref{s:friction_bc}. The second simplification concerns the non-trivial essential
boundary condition $\phi$, which we address in \secref{s:lifting}
by the lifting argument already mentioned in \eqref{apriori}. 
It is customary for the Stokes system
\eqref{eq:stokes_full} to let the pressure $p$ be defined up to a
constant. Less apparent is that the velocity field kernel, namely
  $\nabla\vec u+\nabla\vec u\transpose=0$, is non-trivial if and only if $\Omega$ is  axisymmetric. More importantly, when this kernel is not empty, it is characterized by a small subspace of the rigid body motions, 
     \begin{equation}\label{eq:kernel_z}
         Z\of{\Omega} := \set{\vec z\of{\vec x} = \vec A\vec x + \vec b : \vec x \in \Omega,\, \mat A = -\mat A\transpose \in \bbR^{n\times n},\, \vec b \in \bbR^n,\, \vec z\cdot \vnu\restriction_{\bdy\Omega} = 0} ;
     \end{equation}
see \cite[Appendix A]{MLewicka_SMuller_2011} and \cite[Appendix I]{HBeiraodaVeiga_2004}
for more details.

The standard variational formulation of
\eqref{eq:stokes_full} entails dealing with the Stokes bilinear form
\begin{equation}\label{eq:stokes_form}
       \calS_\Omega\of{\vec u, p}\of{\vec v, q} 
            := \int_\Omega \eta \gradS u: \gradS v - p\divg{\vec v} + q\divg{\vec u}.
\end{equation}
which we obtain from \eqref{eq:stokes} upon formal mutiplication by a test pair
$(\vec v, q)$ and integration by parts. This also leads to the forcing
term
\begin{equation*}
  \calF\of{\vec v, q} 
   := \int_{\Omega}\vec f \cdot \vec v + \int_{\bdy\Omega} \vec\psi \cdot \gamma_0 \vec v + \int_{\Omega} gq,
\end{equation*}
upon invoking \eqref{eq:stokes_bc}. To formulate
\eqref{eq:stokes_full} variationally we need two function spaces.
The first space is that of trial functions $X_r\of\Omega$, which we define as 
    \begin{subequations}\label{eq:banach_space_X_all}
    \begin{equation}\label{eq:banach_space_X}
            X_r\of{\Omega} := V_r\of\Omega \times \LrZ{\Omega} \quad s' \leq r \leq s,
    \end{equation}
with  
$V_r\of\Omega := \set{\vec v \in \sob1r\Omega / Z\of\Omega: \vec v\cdot\vnu = 0}$, $\LrZ{\Omega} := \Lr\Omega / \bbR$ and $1/s + 1/s' = 1$. We have the following characterization of $V_r$ (see \cite[p. 4]{RGDuran_MAMuschietti_2004a}): 
a vector $\vec v \in V_r(\Omega)$ if and only if
\[
   \vec  v \in \Big\{\vec v \in \sob1r\Omega : \vec v \cdot \vec \nu = 0, \ 
                \int_\Omega (\partial_{x_i} v^j - \partial_{x_j} v^i) \ dx = 0 \
               \forall i,j = 1, \dots, n \Big\} .
\]
It follows from its product definition that $X_r\of\Omega$ is complete under the norm  
    \begin{equation}\label{eq:banach_space_X_norm}          
        \norm{(\vec v, p)}_{X_r\of{\Omega}} := \normS{\vec v}1r\Omega + \normLr{p}\Omega.
    \end{equation}

The second function space is that of prescribed data, which we take to be $\dual{X_{r'}\of\Omega}$, the topological dual of $X_{r'}\of\Omega$ where $1/r + 1/r' = 1$. Moreover, $\dual{X_{r'}\of\Omega}$ is complete under the operator norm
    \begin{equation}\label{eq:banach_space_dual_X_norm}
        \norm{\calF}_{\dual{X_{r'}\of\Omega}}
            = \sup_{\norm{\of{\vec v, q}}_{X_{r'}\of{\Omega}} = 1} 
                \abs{\calF\of{\vec v, q}}.
    \end{equation}
\end{subequations}
We note that $\norm{\calF}_{\dual{X_{r'}\of\Omega}}$ is finite
provided $g$ is in $\Lr\Omega$, $\vec f$ belongs to
$\dual{V_{r'}\of\Omega}$, $\vec \psi$ lies in the dual of the trace
space $\gamma_0\of{V_{r'}\of\Omega}$; moreover
all three functions must satisfy the compatibility conditions 
    \begin{equation}\begin{aligned}\label{eq:stokes_compatibility}
 		\int_\Omega g &= 0 , \quad   
 		\int_\Omega \vec f \cdot \vec z + \int_{\bdy\Omega} \vec\psi \cdot \gamma_0 \vec z &= 0 \quad \forall \vec{z} \in Z\of\Omega.
 	\end{aligned}\end{equation}
The variational formulation of the strong equations
\eqref{eq:stokes_full} finally reads: solve
\begin{equation}
       \label{eq:stokes_abstract}
       (\vec u,p) \in X_r(\Omega):
       \quad \calS_\Omega\of{\vec u, p}\of{\vec v, q} 
           = \calF\of{\vec v, q}  \quad \forall (\vec v, q) \in X_{r'}\of{\Omega}.
\end{equation}
With the functional setting in place, we state our main result.
\begin{theorem}[well posedness of \eqref{eq:stokes_full}]\label{T:wellposed}
Let $\Omega$
be a bounded domain of class $W^{2-1/s}_s$ with $s > n$, and let
$s' \leq r \leq s$. For every $\calF \in \dual{X_{r'}\of\Omega}$ there
exists a unique solution $(\vec u, p) \in X_{r}\of{\Omega}$ of
\eqref{eq:stokes_abstract} such that
    \begin{equation}
       \label{eq:stokes_apriori}
       \norm{(\vec u,p)}_{X_r\of{\Omega}} 
           \leq  C_{\Omega, \eta, n, r}\norm{\calF}_{\dual{X_{r'}\of{\Omega}}}.
    \end{equation}
\end{theorem}

We say that the Stokes problem is {\it well-posed} (in the sense of
Hadamard) between the spaces $X_r\of\Omega$ and $\dual{X_{r'}\of\Omega}$
whenever \eqref{eq:stokes_abstract}-\eqref{eq:stokes_apriori} is
satisfied.

The ideas explored in this paper can be summarized as follows.
We develop a new localization technique which features \emph{domain,
space} and \emph{operator} decompositions. Instead of
\eqref{eq:curlf}, we rely solely on
existence and uniqueness of solutions to the Stokes problem in the
whole space $\bbR^n$ for compactly supported data. Finally, we develop
an \emph{index-theoretical} framework to
close the argument \cite[Chapter 27]{PDLax_2002}.
Our approach is general and can be applied to a
wide class of elliptic partial differential equations (PDEs).
After a brief section about notation, 
we split the proof of Theorem \ref{T:wellposed}
into six sections, which describe how the paper is organized:
\begin{description}
    \item[\secref{s:hilbert_case}]
        gives a short proof for the Hilbert space case ($r=2$). 
        The importance of this result is the direct implication of uniqueness for solutions to \eqref{eq:stokes_abstract} when $r \geq 2$ and $\Omega$ is bounded. 
    \item[\secref{s:unbounded}]
        presents fundamental results on well-posedness of the Stokes system
in $\bbR^n$ and the half-space $\halfspace$. These two building blocks are instrumental in constructing a solution of \eqref{eq:stokes_abstract} for $ s' \leq r \leq s$.
     \item[\secref{s:sobolev_piola}]
        constructs a local $W^2_s$ diffeomorphism $\vec\Psi$, whose
        inverse $\vec\Psi^{-1}$ locally flattens $\bdy\Omega$, and analyzes the Piola transform which preserves the essential boundary condition $\vec u \cdot \vnu = 0$.
     \item[\secref{s:sobolev_case}]
         develops a new localization procedure and uses index theory
         to prove the well-posedness of the Stokes system
         \eqref{eq:stokes_full} between the spaces
        $X_r\of\Omega$ and $\dual{X_{r'}\of\Omega}$ for $s' \leq r \leq s$. 
     \item[\secref{s:lifting}]
         deals with the inhomogeneous essential boundary conditions.
     \item[\secref{s:friction_bc}]
         extends the theory to the full Navier boundary condition, i.e. $\beta \neq 0$.
\end{description}

\section{Notation} \label{s:notation}
It will be convenient to distinguish the $n$-th dimension. A vector $\vec x \in \bbR^n$, will be 
denoted by
\[
\vec x = (x^1,\dots,x^{n-1},x^n) = (\vec x',x^n)     , 
\]
with $x^i \in \bbR$ for $i = \overline{1,n}$, $\vec x' \in
\bbR^{n-1}$. We will make a distinction between the reference
  coordinate $(\xe)$ and the physical coordinate $(\vec x)$, such that
  $\vec x = \Psi(\xe)$ where the properties of the map $\Psi$ are listed below.
The symbols $B(\vec x,\delta) \subset \bbR^n$ and $\Be(\xe,\delta) \subset \bbR^n$ will denote the balls of radius $\delta$ centered at $\vec x$ and $\xe$ respectively. Moreover, $\De(\xe',\delta) \subset \bbR^{n-1} \times \{\xe^n\}$ and $\Be_{-}(\xe,\delta) := \Be(\xe,\delta)\cap\bbR^n_{-}$ will be the disc and the lower half ball of radius $\delta$ centered at $\xe'$ and $\xe$ respectively; see \figref{f:RefMap} (right).

\begin{definition}
[$W^{2-1/s}_s$-domain] 
\label{defn:sobolev_domain}
An open and connected set $\Omega$ in $\bbR^n$ is called a $W^{2-1/s}_s$-domain, $s > n$, if at each point $\vec x$ in $\bdy\Omega$ there 
exists  $\delta > 0$ and a function $\omega$ in $\sob{2-1/s}{s,\,\loc}{\bbR^{n-1}}$ such that, after a possible relabeling and reorientation 
of the coordinate axis
    \begin{equation}\label{eq:Omega_lambda}
         \Omega \cap B\of{\vec x, \delta} =
           \set{\vec y = \del{\vec y', y^n} \in B\of{\vec x, \delta} : y^n < \omega \of{\vec y'}} ;
    \end{equation}
see \figref{f:RefMap}. 
A $W^{2-1/s}_s$-domain where $\delta$ can be chosen independently of $\vec x$ is said to be a \emph{uniform} $W^{2-1/s}_s$--domain.
It is easy to verify that every bounded $W^{2-1/s}_s$-domain is a \emph{uniform} $W^{2-1/s}_s$-domain.
\end{definition}

To study the problem \eqref{eq:stokes_full} near $\vec x \in \bdy\Omega$ we need to flatten $\bdy\Omega$ locally. This is realized by a map $\Psi = (\Psi^1,\dots,\Psi^n) : \bbR^n_{-} \rightarrow \bbR^n$ with the following properties (see \figref{f:RefMap}):
\begin{enumerate}[(P1)]
\item \label{P1} $\Psi$ is a diffeomorphism of class $W^2_s$ between $\bbR^n_{-}$ and 
$\Psi(\bbR^n_{-}) \subset \bbR^n$, $\ye \in \Qe \subset \bbR^n_{-}$ denotes the reference coordinate and $\vec y \in \Psi(\bbR^n_{-})$ the physical coordinate so that $\vec y = \Psi(\ye)$ and $Q = \Psi(\Qe)$. 
\item \label{P2} $\Psi(\De(\xe, \delta/2)) = \bdy\Omega \cap B(\vec x, \delta/2)$. 
\item \label{P3} $\Psi(\bbR^n_{-} \setminus \Be(\xe, \delta)) = \calI$ (identity). 
\end{enumerate}
We construct such a map $\Psi$ in Section~\ref{s:loc_diff} but observe now the gain in regularity from $W^{2-\frac{1}{s}}_s$ to $W^2_s$. This regularity improvement is critical for our theory and is achieved by extending the function $\omega$ suitably. 

\begin{figure}[h]
\centering
\includegraphics[width=0.5\textwidth]{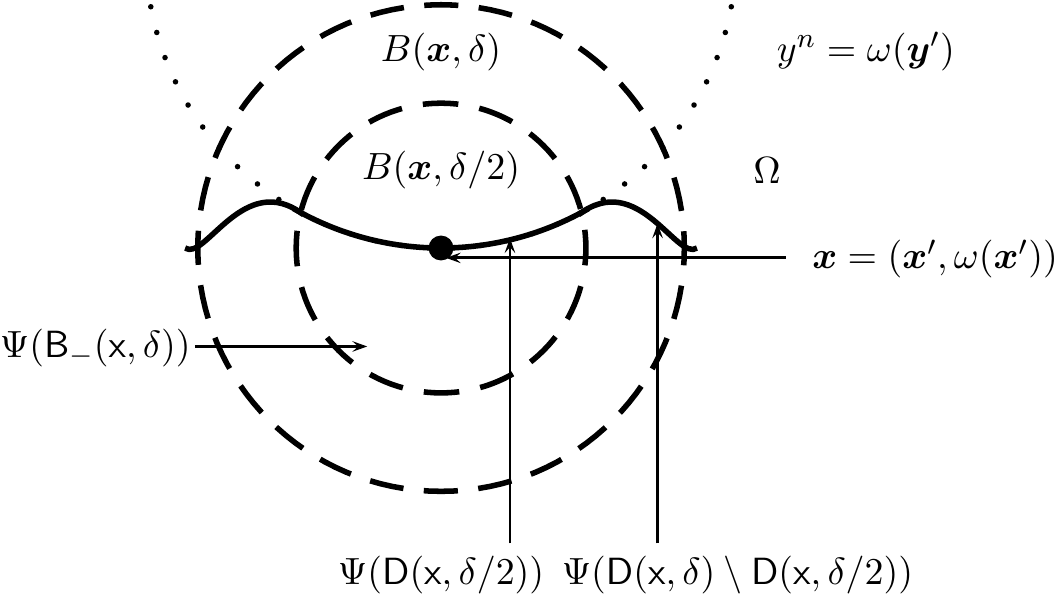} \quad
\includegraphics[width=0.23\textwidth]{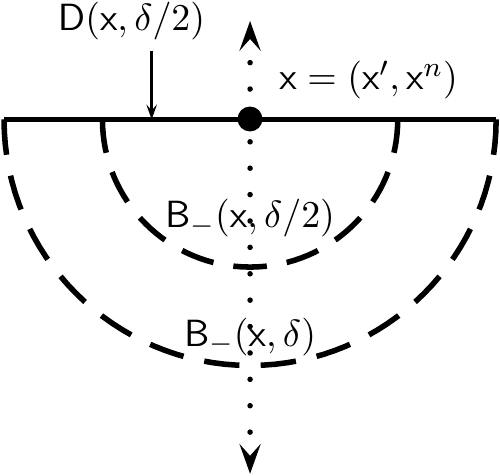}
\caption{The left panel depicts part of the boundary $\bdy\Omega$ which has a graph representation (dotted curve), i.e., if $\vec x \in \bdy\Omega$
  then $\vec x = (\vec x', \omega(\vec x'))$. The domain $\Omega$ is
  assumed to lie below $\bdy\Omega$. The inner dashed curve is $B(\vec x,
  \delta/2)$, and the outer dashed curve is $B(\vec x,\delta)$. 
  The right panel shows the reference domains $\Be_{-}(\xe,\delta/2) =
  \bbR^n_{-}\cap \Be(\xe,\delta/2)$ and $\Be_{-}(\xe,\delta) =
  \bbR^n_{-}\cap \Be(\xe,\delta)$ with discs $\De(\xe,\delta/2)$ and
  $\De(\xe,\delta)$ at the top respectively.
  If $\Psi$ is a local $W^2_s$-diffeomorphism satisfying
    (P\ref{P1})-(P\ref{P3}), we see that $\Psi(\Be_{-}(\xe,\delta))$
    is the part of $\Omega$ lying below the solid curve (left panel).}
\label{f:RefMap}
\end{figure}

If $X_p(\Omega)$ is a Banach space over $\Omega$, we denote by $\norm{\cdot}_{X_p}$ its norm. By $L^p(\Omega)$ with $p \in [1,\infty]$ we denote the space of functions that are Lebesgue integrable with exponent $p$. By $W^k_p(\Omega)$ we denote the classical Sobolev space of functions whose distributional 
derivatives up to $k$-th order are in $L^p(\Omega)$. We indicate with $\sobZ{k}{p}\Omega$ the closure of $C^\infty_c(\Omega)$ in $W^k_p(\Omega)$. The Lebesgue conjugate to $p$ will be denoted by $p'$, i.e. $\frac{1}{p} + \frac{1}{p'} = 1$. 
We denote by $\pair{\cdot,\cdot}$ the duality pairing and sometimes the $L^2$-scalar product when it is clear from the context. 
The relation $a \lesssim b$ indicates that $a \le C b$, with the constant $C$ that does not depend on $a$ or $b$. The 
value of $C$ might change at each occurrence.

Given matrices $\vec P \in \bbR^{n\times n}$, $\vec M = \del{m^{i,k}}_{i,k=1}^n \in \bbR^{n\times n}$, 
and vector $\vec w \in \bbR^n$, we define 
\begin{align}\label{eq:frakP}
\mathfrak{P_{\matP}(\matM)} = \matP \matM \matP^{-1} , 
\end{align}
and note that
\begin{align}\label{eq:gradient_matrix_vector}
    \grad\del{\vec M \vec w} = \grad \vec M \odot \we + \vec M \grad \vec w ,  
\end{align}
where the $(i,j)$ component of the $n \times n$ matrix $\grad \vec M \odot \we$ is 
\[
    \del{\grad \vec M \odot \vec w}^{i,j} = \sum_{k=1}^n \del{\partial_{x^j} m^{i,k}}\vec w^k . 
\]
%

%
%
%
%
%
%
%
%
%

\section{The Hilbert Space Case}
\label{s:hilbert_case}

In this section we prove the well-posedness of the Stokes problem in
$X_2\of{\Omega}\times\dual{X_2\of\Omega}$. Results in this direction
are known for a generalized Stokes system on $C^{1,1}$ domains
\cite{HBeiraodaVeiga_2004}. We show that $\Omega$ being Lipschitz is
sufficient for the homogeneous case $\phi = 0$.
Our proof relies on Korn's inequality, Brezzi's inf-sup condition for saddle-point problems, and Ne\v{c}as' estimate on the right inverse of the divergence operator. 
We collect these results in the sequel.

\begin{proposition}[Korn's inequality]
\label{prop:korn_bounded_domain}
\label{prop:korn_first_inequality} 
Let  $ 1 < r < \infty$ and $\calD$ be a bounded Lipschitz domain in $\bbR^n$. There exists constants $C_1$ and $C_2$ depending only on $\calD, n$ and $r$ such that for every $\vec v$ in $\sob1r\calD$
    \begin{subequations}
    \begin{equation}\label{eq:korn_inequality_first}
        \normS{\vec v}1r\calD \leq C_1\del{\normLr{\vec v}\calD + \normLr{\gradS v}\calD } \leq C_2\normS{\vec v}1r\calD.
    \end{equation}
Moreover, for every $\vec v \in \sob 1r\calD$ there exists a skew symmetric matrix  $\mat A$ in $\bbR^{n\times n}$, and $\vec b\in \bbR^n$ such that 
    \begin{equation}\label{eq:korn_inequality_bdd}
        \normS{\vec v - \of{\vec A\vec x + \vec b}}1r\calD \leq C_{\calD, n, r}\normLr{\gradS v}\calD.
    \end{equation}
    \end{subequations}
\end{proposition}
\begin{proof} See \cite[Theorem A.1]{MLewicka_SMuller_2011} for $r = 2$, \cite[Section 2]{RGDuran_MAMuschietti_2004a} for $1 < r < \infty$ in bounded domains.
\end{proof}

\begin{lemma}[equivalence of norms]
\label{lem:equiv_norm}
Let  $ 1 < r < \infty$ and $\calD$ be a bounded Lipschitz domain. For every $\vec v$ in $V_r\of\calD$ the following holds
    \[
       \normS{\vec v}1r\calD \leq C_{\calD, n, r}\normLr{\gradS v}\calD.
    \]
\end{lemma}
\begin{proof} This proceeds by contradiction. However, the
  argument is fairly standard, is based on \cite[Theorem
    A.2]{MLewicka_SMuller_2011} for $r=2$, and is thus omitted. 
\end{proof}

\begin{remark}[boundary condition $\vec z \cdot \vec \nu = 0$] \rm
\label{rem:norm_equiv}
The difference between \lemref{lem:equiv_norm} and
Proposition~\ref{prop:korn_bounded_domain} is that the vector-fields
$\vec z$ in $Z\of{\calD}$ satisfy $\vec z \cdot \vnu = 0$ while the
ones from Korn's inequality do not have this requirement. We further
remark that \lemref{lem:equiv_norm} remains valid if the condition
$\vec v \cdot \vec \nu = 0$ is imposed only on a subset of $\bdy\calD$
with positive measure. We will use this result in
\thmref{thm:bhs_isomorphic_domains}, \lemref{lem:C_compact} and
\lemref{lem:boundary_regularity} below.
\end{remark}

Next we state Brezzi's characterization of \eqref{eq:stokes_abstract} as a saddle-point problem. We rewrite \eqref{eq:stokes_abstract} as follows: find a unique $\of{\vec u, p}$ in $V_r\of\Omega \times \LrZ\Omega$ such that
    \begin{equation}\label{s:stokes_saddle_pt}\begin{aligned}
        \eta\pair{\gradS u, \gradS v}_\Omega - \pair{p, \divg{\vec v}}_\Omega &= \calF\of{\vec v,0} &&\forall \vec v \in V_{r'}\of\Omega, \\
        \pair{\divg{\vec u}, q}_\Omega &= \calF\of{\vec 0,q} &&\forall q \in \LrpZ\Omega.
    \end{aligned}\end{equation}

\begin{lemma}[inf-sup conditions]\label{lem:brezzi} The saddle point problem \eqref{s:stokes_saddle_pt} is well-posed in $(V_r\of\Omega \times \LrZ\Omega)\times (\dual{V_{r'}\of\Omega}\times \dual{\LrpZ\Omega})$ if and only if there exist constants $\alpha, \beta > 0$ such that
    \begin{subequations} \label{eq:brezzi_conditions}
    \begin{align}
      \label{eq:brezzi_a}
        \inf_{\vec w \in \mathring V_r}\sup_{\vec v \in \mathring V_{r'}} 
            \frac{\pair{\gradS w, \gradS v}}{\norm{\vec w}_{V_r}\norm{\vec v}_{V_{r'}}} 
        = \inf_{\vec v \in \mathring V_{r'}}\sup_{\vec w \in \mathring V_{r}} 
            \frac{\pair{\gradS w, \gradS v}}{\norm{\vec w}_{V_r}\norm{\vec v}_{V_{r'}}} 
        = \alpha &> 0, \\
        \label{eq:brezzi_b}
        \inf_{q \in L^{r'}_0}\sup_{\vec w \in V_{r}} \frac{\pair{\divg{\vec w}, q}}{\norm{\vec w}_{V_r}\norm{q}_{L^{r'}}} = \beta &> 0,
    \end{align}
    \end{subequations}
where $\mathring V_{r} := \set{\vec w \in V_r\of\Omega : \pair{\divg{\vec w}, q} = 0,\, \forall q \in \LrpZ\Omega}.$
In addition, there exists $\gamma = \gamma (\alpha, \beta, \eta)$ such that the solution $(\vec u, p)$ is bounded by
   \[
       \norm{\of{\vec u,p}}_{X_r\of\Omega} \leq \gamma \norm{\calF}_{\dual{X_{r'}\of\Omega}}.
   \]
\end{lemma}
\begin{proof} See \cite[Section II.1, Theorem 1.1]{FBrezzi_MFortin_1991a}.
\end{proof}

\begin{theorem}
[well-posedness for $r=2$] 
\label{thm:hilbert_case}
Let $\Omega$ be a bounded Lipschitz domain. The Stokes problem \eqref{eq:stokes_abstract} is well-posed in $X_{2}\of\Omega\times\dual{X_{2}\of\Omega}$.
\end{theorem}
\begin{proof} 
It suffices to check Brezzi's conditions \eqref{eq:brezzi_conditions}.
\end{proof}

\begin{remark}[boundary regularity] \rm
The Lipschitz regularity of $\partial\Omega$ is adequate only for
$\phi=0$ and $r=2$, as alluded to in the introduction. In general, we
need to work with domains of class $W^{2-1/s}_s$.
\end{remark}

\section{Stokes Problem on Unbounded Domains ($\bbR^n$ and $\halfspace$) }
\label{s:unbounded}
The purpose of this section is to prove the existence, uniqueness and \emph{local} regularity of the Stokes problem \eqref{eq:stokes_abstract} in the whole space $\bbR^n$ and the half-space $\halfspace$ for data with \emph{compact} support. These two problems are the essential building blocks for the localization procedure in \secref{s:sobolev_case}. 
This problem has been extensively studied under different functional frameworks; we refer to \cite[Introduction]{FAlliot_CAmrouche_1999a} for an overview. 



Weighted Sobolev spaces are an extremely general framework for it provides a wealth of predictable behaviors at $\infty$ when considering different weight functions. 
A different framework is the one of Homogeneous Sobolev spaces, its main disadvantage being the lack of control on the $L^r$-norm of the function. 
Fortunately, these two frameworks are interchangeable as long as the data in question has compact support  and one is not interested in the behavior at $\infty$ of the functions being analyzed \cite[Proposition 4.8]{FAlliot_CAmrouche_1999a}. 

With this equivalence in hand, and the fact that our work was originally inspired by that of Galdi-Simader-Sohr \cite{GPGaldi_CGSimader_HSohr_1994}, we choose to work with the Galdi-Simader's characterization for homogeneous Sobolev spaces \cite{GPGaldi_CGSimader_1990}. The rest of this section is split into three parts. In \secref{s:homogeneous_sobolev_spaces} we recall this essential characterization and define the equivalent $X_r\of{\calD}$ spaces for unbounded domains $\calD$. In \secref{s:whole_space} we prove the well-posedness of the Stokes problem in its symmetric gradient form in $\bbR^n$. Finally, in \secref{s:half_space} we extend the result to the half-space $\halfspace$.


\subsection{Homogeneous Sobolev Spaces}
\label{s:homogeneous_sobolev_spaces}
The solution space $X_r(\Omega)$ is too small to prove an existence and uniqueness result for unbounded domains \cite[Section 2]{GPGaldi_CGSimader_HSohr_1994}. In these cases we are led to consider the homogeneous Sobolev spaces 
	\begin{equation}\label{eq:gs_space}\begin{aligned}
		\sobH 1 r {\bbR^{n}} = \sobHZ 1 r {\bbR^n} 
		    &:= \clos{[\CCinf{\bbR^{n}}]^n}^{\normSZ{\cdot}1r{\bbR^{n}}}, \\
		\sobH 1 r {\halfspace} 
		    &:= \clos{[\CCinf{\clos{\bbR^{n}_{-}}}]^n}^{\normSZ{\cdot}1r {\bbR^n_{-}}}, \\ 
		\sobHZ 1 r {\bbR^{n}_{-}} 
		    &:= \clos{[\CCinf{\clos{\bbR^n_{-}}}]^{n-1}\times \CCinf{\bbR^n_{-}}}^{\normSZ{\cdot}1r {\bbR^n_{-}}},
	\end{aligned}\end{equation}
where $\CCinf{\calD}$ are $C^\infty$ functions with compact support in $\calD$, the half-space $\halfspace$ is given by 
$\vec x = (\vec x', x^n) \in \bbR^n$ with $x^n < 0$, and $\clos\halfspace = \halfspace \cup \bdy\halfspace$ with 
$\vec x$ in $\bdy\halfspace$ if and only if $x^n = 0$. The statement 
$\vec v = (\vec v', v^n) \in [\CCinf{\clos{\halfspace}}]^{n-1} \times \CCinf{\halfspace}$ implies 
$\vec{v}\restriction_{\bdy\halfspace} = (\vec v'\restriction_{\bdy\halfspace}, 0)$ for $\vec v' \in [\CCinf{\clos{\halfspace}}]^{n-1}$. 
For a detailed presentation of these spaces, their duals and trace spaces see \cite[Chapter II]{GPGaldi_2011a}, in particular \cite[Theorem II.10.2]{GPGaldi_2011a} for the trace space results.

Next we recall a result by Galdi-Simader on the characterization of $\sobH 1r\calD$ and $\sobHZ 1r\calD$ with $\calD$ equal to $\bbR^n$ or $\halfspace$  \cite[Lemma 2.2]{HKozono_HSohr_1992},\cite[Section 1]{HBeiraodaVeiga_2005}, \cite{GPGaldi_CGSimader_1990}, and \cite{GPGaldi_CGSimader_HSohr_1994}.

\begin{proposition}[Galdi-Simader]\label{prop:galdi_simader}
Let $1 < r < \infty$ and $\sobH 1r\calD$ and $\sobHZ 1r\calD$ be the spaces defined in \eqref{eq:gs_space}. The following characterization holds, 
	\begin{equation}\label{eq:galdi_simader_a}\begin{aligned}
		\sobH 1 r {\bbR^n} 
		    &= \set{\eclass{\vec v}_1 \in [\Lrloc{\bbR^n}]^n : 
		            \grad \vec v \in \sbr{\Lr{\bbR^n}}^{n\times n}}, \\
		\sobHZ 1 r \halfspace
		    &= \set{\vec v = (\eclass{\vec v'}_1,v^n) \in [\Lrloc{\clos\halfspace}]^n : 
		            \grad \vec v \in \Lr\halfspace^{n\times n},\, v^n\restriction_{\bdy\halfspace} = 0},
	\end{aligned}\end{equation}
where $\eclass{\vec v}_1$ is the equivalence class of functions in $\sbr{\Lrloc{\calD}}^n$ which differ by a constant vector. Furthermore, if $ 1 < r < n$ then additionally 
	\begin{equation}\label{eq:galdi_simader_b} \begin{aligned}
		\sobH 1r\calD
		    &= \set{\vec v \in [L^{r^*}_\loc\of{\calD}]^n:  
		            \grad \vec v \in [\Lr\calD]^{n\times n}}, \quad \calD = \bbR^n \mbox{ or } \halfspace \\
		\sobHZ 1r\halfspace 
		    &= \set{\vec v \in 
		            \of{[L^{r^*}_\loc\of\halfspace]^{n-1} \times L^{r^*}\of\halfspace} 
		            \cap \sobH 1r\halfspace: 
		             v^n\restriction_{\bdy\halfspace} = 0},
	\end{aligned}\end{equation}
where $r^*$ is the Sobolev conjugate of $r$ and is given by $1/{r^*} = 1/r - 1/n$.
Moreover, using Gagliardo-Nirenberg inequality we have for $1 < r < n$
\begin{equation}\label{eq:GN_inequality}
    \norm{\vec v}_{L^{r^*}\del{\calD}} \lesssim \norm{\grad \vec v}_{L^r\del{\calD}} . 
\end{equation}
\end{proposition}

We conclude by introducing the functional space $X_r\of\calD$ when $\calD$ is $\bbR^n$ or $\halfspace$. The distinction between this and  \eqref{eq:banach_space_X} is that now we use the homogeneous Sobolev spaces defined above, and the pressure space is simply $\Lr\calD$, i.e. 
    \begin{subequations}\label{eq:banach_space_X_hom_all}
    \begin{equation}\label{eq:banach_space_X_hom}
            X_r\of{\calD} := V_r\of\calD \times \Lr{\calD}\quad   1 < r < \infty,
    \end{equation} 
with 
$V_{r}\of{\bbR^n} = \sobH{1}{r}{\bbR^n}$, and $V_{r}\of\halfspace= \sobHZ{1}{r}\halfspace$. 

It follows from the product definition of $X_r\of\calD$ that it is a complete space under the 
norm  
    \begin{equation}\label{eq:banach_space_X_snorm}          
        \abs{(\vec v, p)}_{X_r\of{\calD}} := \normSZ{\vec v}1r\calD + \normLr{p}\calD.
    \end{equation}
The space for the prescribed data, is $\dual{X_{r'}\of\calD}$, the topological dual of $X_{r'}\of\calD$,  where $1/r + 1/r' = 1$. Moreover, $\dual{X_{r'}\of\calD}$ is complete under the operator norm
    \begin{equation}\label{eq:banach_space_dual_X_snorm}
        \norm{\calF}_{\dual{X_{r'}\of\calD}}
            = \sup_{\abs{\of{\vec v, q}}_{X_{r'}\of{\calD}} = 1} 
                \abs{\calF\of{\vec v, q}}.
    \end{equation}
    \end{subequations}

\subsection{Stokes Problem in $\bbR^n$}
\label{s:whole_space}
In this section we investigate the well-posedness of the Stokes problem \eqref{eq:stokes_abstract} between the spaces $ X_{r}\of{\bbR^n}$ and $\dual{ X_{r'}\of{\bbR^n}}$. We begin by recalling the well posedness of the usual Stokes problem without symmetric gradient \cite[Section IV.2]{GPGaldi_2011a}. 

\begin{lemma}[Well posedness in $\bbR^n$]
\label{eq:rn_wellposed_sohr} 
Let $1 < r < \infty$, $n \geq 2$. For each $\vec f \in \sobH{-1}{r}{\bbR^n} = \dual{\sobH{1}{r'}{\bbR^n}}$ and $g \in \Lr{\bbR^n} = \dual{\Lrp{\bbR^n}}$, there exists a unique pair $\del{\vec u, p} \in \sobH{1}{r}{\bbR^n} \times \Lr{\bbR^n}$ satisfying
	\[
		-\lapl\vec u + \grad p = \vec f,\quad \divg{\vec u} = g \quad \mbox{ in }\bbR^n
	\]
in the sense of distributions, which depends continuously on the data, i.e.
	\[
		\normSZ{\vec u}1r{\bbR^n} + \normLr{p}{\bbR^n} \leq C_{n,r}\del{\normSH{\vec f}{-1}r{\bbR^n} + \normLr{g}{\bbR^n}}.
	\]
Additionally, if $1 < t < \infty$, $\vec f \in \sobH{-1}{t}{\bbR^{n}}$ and $g \in L^{t}\of{\bbR^n}$, then $\vec u \in \sobH1{r}{\bbR^n} \cap \sobH1{t}{\bbR^n}$ and $p \in \Lr{\bbR^n} \cap  \times L^{t}\of{\bbR^n}$.
\end{lemma}

This result, combined with the identity $\divg{\grad \vec u\transpose} = \grad\divg{\vec u}$, yields the well-posedness of \eqref{eq:stokes_abstract}-\eqref{eq:stokes_apriori} in $\bbR^n$.

\begin{theorem}[well-posedness of $\calS_{\bbR^n}$]
\label{thm:rn_wellposed}
Let $1 < r < \infty$, $n \geq 2$. The Stokes problem $\calS_{\bbR^n}(\vec u,p) = \calF$ is well-posed in the space pair $\del{X_r\of{\bbR^n} , \dual{X_{r'}\of{\bbR^n}}}$, namely \eqref{eq:stokes_abstract}-\eqref{eq:stokes_apriori} are satisfied. Additionally, if $1 < t < \infty$ and $\calF \in \dual{X_{t'}\of{\bbR^n}}$, then $\of{\vec u, p} \in X_{t}\of{\bbR^n}$.
\end{theorem}

\subsection{Stokes Problem in $\bbR^n_{-}$}
\label{s:half_space}
In this section we show the well-posedness of the Stokes problem \eqref{eq:stokes_abstract} in the space pair $\del{X_{r}\of{\halfspace} ,  \dual{X_{r'}\of{\halfspace}}}$. Although the reflection technique employed is well-known, the construction sets the stage for the localization section. A very general result in this direction is the work by Beir\~ao da Veiga-Crispo-Grisanti \cite{HBeiraodaVeiga_FCrispo_CRGrisanti_2011}. 

\begin{theorem}
[well-posedness of $\calS_{\halfspace}$]
\label{thm:hs_wellposed}
Let $1 < r < \infty$, $n \geq 2$. The Stokes problem \eqref{eq:stokes_abstract} is well-posed from $X_r\of{\halfspace}$ to $\dual{X_{r'}\of{\halfspace}}$. Additionally, if $1 < t < \infty$ and $\calF \in \dual{X_{t'}\of{\halfspace}}$, then $\of{\vec u, p} \in X_{t}\of{\halfspace}$.
\end{theorem}
\begin{proof}
In view of the uniqueness results in \cite[Theorem 3.1]{RFarwig_1994}, it suffices  to construct a solution to the Stokes problem in the half-space which depends continuously on the data.

Define for each function $\hat\varphi:\bbR^n \to \bbR$ its upper and lower parts as 
    \[
        \varphi_+\of{\vec x} := \hat\varphi\of{\vec x', -x^n}, \quad \varphi_{-}\of{\vec x} := \hat\varphi\of{\vec x} \quad \mbox{ for all } \vec x \in \halfspace.
    \]
Take $\of{\vecH v, \hat q} \in X_{r'}\of{\bbR^n}$ and define their pullbacks into $\halfspace$ as follows: 
    \[
        \vec v := \frac{1}{2}\del{\vecH{ v}'_{-} + \vecH{ v}'_+, \hat v^n_{-} - \hat v^n_+ }\transpose, \quad  q := \frac{1}{2}\del{\hat q_{-} + \hat q_{+}}.
    \]
It is simple to show that $\of{\vec v, q} \in X_{r'}\of\halfspace$ and 
    \[
        \abs{\of{\vec v, q}}_{X_{r'}\of\halfspace} \leq \abs{\of{\vecH{ v}, \hat q}}_{X_{r'}\of{\bbR^n}}.
    \]

Let $\calF \in \dual{X_{r'}\of\halfspace}$ be fixed but arbitrary and define $\hat\calF\of{\vecH{ v}, \hat q} := \calF\of{\vec v, q}$. It follows immediately from our current results that $\hat\calF$ is a linear functional on $X_{r'}\of{\bbR^n}$ and $\| \hat\calF \|_{\dual{X_{r'}\of{\bbR^n}}} \leq \|\calF\|_{\dual{X_{r'}\of{\halfspace}}}$. Therefore, \thmref{thm:rn_wellposed} for $\bbR^n$ asserts the existence and uniqueness of a solution $\of{\hat{\vec w}, \hat\pi}$ in $X_{r}\of{\bbR^n}$ to \eqref{eq:stokes_abstract} with the forcing function $\hat\calF$, i.e. 
    \begin{equation}\label{eq:hs_on_rn}
       \calS_{\bbR^n}\of{\hat{\vec w}, \hat\pi}\of{\vecH{ v}, \hat q} = \hat\calF\of{\vecH{ v}, \hat q}  \quad \forall (\vecH{ v}, \hat q) \in X_{r'}\of{\bbR^n},
    \end{equation}
and $\norm{\of{\hat{\vec w}, \hat\pi}}_{X_{r}\of{\bbR^n}} \leq C_{n,r}  \norm{\calF}_{\dual{X_{r'}\of{\halfspace}}}$.

Since the test functions $\del{\vecH{v},\hat{q}}$ are arbitrary, we take $(\vec v, q) \in X_{r'}\of{\halfspace}$ and test \eqref{eq:hs_on_rn} with $(\vecH{ v}, \hat q)$ defined as even reflections for $q$ and $\vec v'$, and an odd reflection for $v^n$, i.e.
    \[
        \hat q_+ = \hat q_{-} = q,
        \quad \vecH{ v}'_+ = \vecH{ v}'_{-} = \vec v',
        \quad \hat v^n_{+} = -\hat v^n_{-} = -v^n. 
    \]
We can immediately verify that $\hat\calF(\vecH{ v}, \hat q) = \calF\of{\vec v, q}$ 
still holds, and after some technical computations, 
	\[
		\calS_{\bbR^n}\of{\hat{\vec w}, \hat\pi}\of{\vecH{ v}, \hat q}
		 	= 2\calS_{\halfspace}\of{\vec u, p}\of{\vec v, q},
	\]
where $p\of{\vec x} = \frac{1}{2}\del{\hat\pi\of{\vec x} + \hat\pi\of{\vec x', -x^n}}$ and $\vec u = \del{\vec u', u^n}$ is given by  
	\[
		\vec u'\of{\vec x} = \frac{1}{2}\del{\hat{\vec{w}}'\of{\vec x} + \hat{\vec{w}}'\of{\vec x', -x^n}},
		\quad  u^n\of{\vec x} = \frac{1}{2}\del{ \hat w^n\of{\vec x} - \hat w^n\of{\vec x', -x^n}}.
	\]
Finally, $\vec u\cdot\vnu\of{\vec x', 0} = u^n\of{\vec x', 0} = 0$ and $(\vec u, p)$ also satisfies estimate \eqref{eq:stokes_apriori}. 
This concludes the proof.
\end{proof}

%


\section{Sobolev Domains and the Piola Transform}
\label{s:sobolev_piola}

\subsection{Local Diffeomorphism on Sobolev Domains} 
\label{s:loc_diff}
We begin with a definition of the local-diffeomorphism.
\begin{definition}[diffeomorphism]
\label{def:loc_diff}
A map $\Psi : \calU \rightarrow \calV$, where $\calU$, $\calV$ are open subsets of $\bbR^n$, is a $W^2_s$-diffeomorphism if $\Psi$ is of class $W^2_s$, is a bijection and $\Psi^{-1}$ is of class $W^2_s$. A map $\Psi$ is a local $W^2_s$-diffeomorphism if for each point $\xe \in \calU$ there exists an open set $\Ue \subset \calU$ containing $\xe$, such that $\Psi(\Ue) \subset \calV$ is open and $\Psi : \Ue \rightarrow \Psi(\Ue)$ is a $W^2_s$-diffeomorphism. 
\end{definition}

Given a $W^{2-1/s}_s$-domain $\Omega$ it is essential to extend the local graph representation $\omega$ in Definition~\ref{defn:sobolev_domain} to a smooth, open, bounded subset of $\bbR^n$ with the extension being $W^2_s$--regular. 

\begin{definition}
[bubble domain] 
\label{defn:bubble_domain}
An open set $\Theta\of{\xe, \delta}$ is called a \emph{bubble domain} of size $\delta$ if its $C^{\infty}$ boundary is obtained by smoothing the ``corners'' of the lower half-ball $\Be_{-}(\xe,\frac{3}{2}\delta) = \Be\del{\xe, \frac{3}{2}\delta} \cap \halfspace$ so that  
    \[
        \clos{\Be_{-}\of{\xe, \delta}} \subsetneq \clos{\Theta\of{\xe, \delta}} \subsetneq \clos{\Be_{-}\of{0, \frac{3}{2}\delta}}.
    \]
This is depicted in \figref{f:internal_estimate}.
\end{definition}

Our strategy consists of the following four steps:
\begin{itemize}
\item[Step 1.] In Lemma~\ref{lem:omega_est} we prove useful norm estimates for 
     $\omega \in \sob{2-1/s}{s}{\De(\xe',\delta)}$, where $\De(\xe',\delta)$ is an open disc of radius $\delta$ 
     centered at $\xe'$. 

\item[Step 2.] It is not possible to invoke a standard extension of $\omega$ to the bubble domain 
$\Theta(\xe, \delta)$ and still arrive at $W^2_s$-regularity for this extension. Therefore we introduce a smooth 
characteristic function $\varrho$ in Definition~\ref{def:smooth_char}, which permits us to define a compactly 
supported function $\calC\omega$ on a disc $\De(\xe',\delta)$ in Lemma~\ref{lem:omega_compact_support} and to 
prove norm estimates for $\calC\omega$ in terms of $\omega$. 
     
\item[Step 3.] We define an harmonic extension $\calE \omega$ of $\calC\omega$ to the bubble domain 
$\Theta(\xe, \delta)$ and prove the $W^2_s$-norm estimates for $\calE \omega$ in terms of $\omega$ in 
Lemma~\ref{lem:Hext}. 
     
\item[Step 4.] 
We define an extension $\widetilde{\calE}\omega := \varrho \calE \omega$ upon
multiplying $\calE \omega$ by a cutoff function $\varrho$ which is equal 1 in
$\Be(\xe,\delta/2)$ and vanishes outside $\Be(\xe,\delta)$ (see
\eqref{eq:tildeE}), and prove its $W^2_s$-regularity in
\lemref{lem:tcalE}.

\end{itemize}

This allows us to define $\Psi : \bbR^n_{-} \rightarrow \bbR^n$ as follows and prove that $\Psi$ is $W^2_s$-regular in Corollary~\ref{cor:omega_diffeomorphism}:
\begin{equation}\label{eq:loc_diff}
\left\{\begin{array}{l}
           \vec x' = \xe ' = \Psi'(\xe) \\
               x^n = \xe^n + \widetilde\calE\omega(\xe) = \Psi^n(\xe) , 
       \end{array}
\right.    
\end{equation}
where $\widetilde\calE\omega$ is the extension of Step 4.
We remark that the mapping 
$\vec x \mapsto \Psi^{-1}(\vec x) = \xe$ ``straightens out $\bdy\Omega$"; see \figref{f:RefMap}. 
The construction of such a $\Psi$ which is $W^2_s$--regular, on a
Sobolev domain which is only $W^{2-1/s}_s$--regular,  is one of the
key contributions of this paper.
We point out that a standard diffeomorphism will only be
$W^{2-1/s}_s$--regular,
which is not sufficient for our purposes.

Given $\vec x \in \bdy\Omega$, we rotate $\Omega$ about $\vec x = \xe'$
  so that $\omega$ in
  Definition~\ref{defn:sobolev_domain} satisfies additionally
    \begin{subequations}
    \begin{equation} \label{eq:omega_sobolev_gradient}
        \grad_{\xe'}\omega\of{\xe'} = 0.
    \end{equation}  
We further translate $\Omega$ in the $\vec e_n$ direction to enforce
    \begin{equation} \label{eq:omega_sobolev_integral}
        \fint_{\De\of{\xe', \delta}} \omega = 0.    
    \end{equation}
    \end{subequations}

\begin{lemma}[properties of $\omega$]
\label{lem:omega_est}
Let $\Omega$ be a domain of class $W^{2-1/s}_s$, $\vec x \in \bdy\Omega$, and $\omega$ satisfy \eqref{eq:omega_sobolev_gradient}-\eqref{eq:omega_sobolev_integral}. Then
    \begin{subequations}
    \begin{align}
        \norm{\omega}_{L^\infty(\De(\xe', \delta))} 
        &\lesssim \delta^{2-n/s} \normSZ{\omega}{2-1/s}s{\De(\xe', \delta)}  \label{eq:Linfomega} , \\
        \norm{\grad_{\xe'}\omega}_{L^\infty(\De\del{\xe',{\delta}})} 
        &\lesssim  \delta^{1-n/s} \normSZ{\omega}{2-1/s}s{\De\del{\xe',{\delta}}}  \label{eq:SemiLinfomega} ,  \\
        \normSZ{\omega}{1-1/s}s{\De(\xe', \delta)} 
        &\lesssim \delta^{n/s} \norm{\grad_{\xe'}\omega}_{L^\infty(\De\del{\xe',{\delta}})}   \label{eq:Semiomega} . 
    \end{align}
    \end{subequations}
\end{lemma}
\begin{proof}
We proceed in three steps. For simplicity we set $\De_{\delta} := \De\of{\xe', \delta}$ and $\grad' := \grad_{\xe'}$.

\noindent \boxed{1} Invoking \eqref{eq:omega_sobolev_integral}, in conjunction with Poincar\'e inequality we deduce 
      \[
         \norm{\omega}_{L^\infty(\De_{\delta})} \lesssim \delta \norm{\grad' \omega}_{L^\infty(\De_{\delta})} . 
      \]

\noindent \boxed{2} Expression  \eqref{eq:omega_sobolev_gradient} implies 
      \begin{align*}
          \abs{\grad'\omega(\ye')} 
          = \abs{\grad'\omega(\ye')-\grad'\omega( \xe')} 
              \lesssim \delta^{1-n/s} \abs{\omega}_{C^{1,1-n/s}(\De_{\delta})}
          \lesssim \delta^{1-n/s}
          \, \abs{\omega}_{W^{2-1/s}_s(\De_{\delta})}  ,
      \end{align*}
      because $\sob{2-1/s}s{\De_{\delta}} \subset C^{1,1-n/s}(\De_{\delta})$. We
      thus deduce \eqref{eq:SemiLinfomega}, and 
      combining this with \boxed{1} we obtain \eqref{eq:Linfomega}. 
      
\noindent \boxed{3} To prove \eqref{eq:Semiomega}, we use the definition of the $W^{1-1/s}_s$ seminorm, i.e.
      \[
          \normSZ{\omega}{1-1/s}s{\De_{\delta}}^s 
          = \int_{\De_{\delta}}\int_{\De_{\delta}} 
            \frac{\abs{ \omega\of{\xe'} - \omega\of{\ye'}}}{\abs{\xe' - \ye'}^{s+n-2}}^s 
                \dif\xe' \dif\ye'  ,    
      \]
      because the exponent of the denominator is $(n-1)+s(1-\frac{1}{s}) = s+n-2$. 
      A direct estimate further yields
      \[
          \normSZ{\omega}{1-1/s}s{\De_{\delta}}^s       
          \le \norm{\grad' \omega}_{L^\infty(\De_{\delta})}^s \int_{\De_{\delta}}\int_{\De_{\delta}}  
          \frac{1}{\abs{\xe' - \ye'}^{n-2}} \dif\xe' \dif\ye' ,
      \]
      and the double integral is of order $\delta^n$. This implies                   
      \[
          \normSZ{\omega}{1-1/s}s{\De_{\delta}}^s \lesssim \delta^n \norm{\grad'\omega}_{L^\infty(\De_{\delta})}^s  
      \]  
     which is the asserted bound \eqref{eq:Semiomega}. 
\end{proof}

\begin{definition}[cutoff function] 
\label{def:smooth_char}
Given $\vec x \in \bbR^n$, a function $\varrho$ in $\CCinf{\bbR^n}$ such that $\varrho = 1$ in $B\of{\vec x , \delta/2}$, 
$0 \leq \varrho \leq 1$ in $B\of{\vec x , \delta} \setminus B\of{\vec x , \delta/2}$ and $\varrho = 0$ in $\bbR^{n}\setminus B\of{\vec x,\delta}$ will be called a \emph{cutoff function} of $B\of{\vec x , \delta/2}$.
\end{definition}
\begin{lemma}
[compactly supported graph] 
\label{lem:omega_compact_support}
Let $\Omega$ be a domain of class $W^{2-1/s}_s$, $\vec x \in \bdy\Omega$, and $\omega$ satisfy \eqref{eq:omega_sobolev_gradient}-\eqref{eq:omega_sobolev_integral}. Then $\calC\omega(\ye') := \varrho(\ye',\omega(\ye')) \omega(\ye')$ for $\ye' \in \De(\xe',\delta)$ satisfies 
$\calC\omega \in \sob{2-1/s}s{\De(\xe',\delta)}$, 
$\calC\omega\restriction_{\De(\xe',\delta/2)} = \omega$, and 
    \begin{equation}
    \label{eq:bhs_diffeomorphism_bound_a}
        \normS{\calC\omega}{2-1/s}s{\De\of{\xe', \delta}} 
            \lesssim \normS{\omega}{2-1/s}s{\De\of{\xe', \delta}}.
    \end{equation}
\end{lemma}
\begin{proof}
We proceed in four steps. For simplicity we again define 
$\De_{\delta} = \De\of{\xe', \delta}$ and $\grad' = \grad_{\xe'}$. 

\noindent \boxed{1} As $\grad'\calC\omega = (\grad' \varrho) \omega + \varrho \grad' \omega$, 
we deduce 
\begin{align*}
\norm{\grad' \calC\omega}_{L^s(\De_{\delta})} 
&\le \norm{\grad'\varrho}_{L^\infty(\De_{\delta})} \norm{\omega}_{L^s(\De_{\delta})}     + 
\norm{\varrho}_{L^\infty(\De_{\delta})} \norm{\grad'\omega}_{L^s(\De_{\delta})}  .  
\end{align*}    
Since $\norm{\grad'\varrho}_{L^\infty(\De_{\delta})} \lesssim \delta^{-1}$ and $\norm{\omega}_{L^s(\De_{\delta})} \lesssim \delta \norm{\grad'\omega}_{L^s(\De_\delta)}$ (Poincar\'e inequality), we infer that $\norm{\grad' \calC\omega}_{L^s(\De_\delta)} \lesssim \norm{\omega}_{W^1_s(\De_\delta)}$. 

\noindent \boxed{2} Invoking the definition of the $W^{1-1/s}_s$-seminorm, we get
\begin{align*}
\normSZ{\grad' \calC\omega}{1-1/s}s{\De_{\delta}}^s 
&= \int_{\De_{\delta}}\int_{\De_{\delta}} 
\frac{\abs{\grad'\calC \omega\of{\xe'} - \grad' \calC\omega\of{\ye'}}}{\abs{\xe' - \ye'}^{s+n-2}}^s 
\dif\xe' \dif\ye'    \\
&\le \int_{\De_{\delta}}\int_{\De_{\delta}} 
\frac{\abs{\omega\of{\xe'} \grad'\varrho\of{\xe'}  
- \omega\of{\ye'} \grad'\varrho\of{\ye'} }}{\abs{\xe' 
- \ye'}^{s+n-2}}^s \dif\xe' \dif\ye'   \\
&\quad + \int_{\De_{\delta}}\int_{\De_{\delta}} 
\frac{\abs{\varrho\of{\xe'}\grad'\omega\of{\xe'} 
- \varrho\of{\ye'}\grad'\omega\of{\ye'}}}{\abs{\xe' - \ye'}^{s+n-2}}^s 
\dif\xe' \dif\ye'                                    
= \textrm{I} + \textrm{II} . 
\end{align*}
We now proceed to tackle \textrm{I} and \textrm{II} separately. 

\noindent \boxed{3} We estimate the term \textrm{I} as follows: we add
and subtract $\omega(\xe') \grad'\varrho(\ye')$ to obtain 
\begin{align*}
\textrm{I} 
&\lesssim \int_{\De_{\delta}}\int_{\De_{\delta}}
\frac{\abs{\grad'\varrho\of{\xe'} 
- \grad'\varrho\of{\ye'} }^s \abs{\omega\of{\xe'}}^s}{\abs{\xe' 
- \ye'}^{s+n-2}} \dif\xe' \dif\ye'                 \\
&\quad +\int_{\De_{\delta}}\int_{\De_{\delta}}
\frac{\abs{\grad'\varrho\of{\ye'}}^s \abs{\omega\of{\xe'} - \omega(\ye')}^s}{\abs{\xe' 
- \ye'}^{s+n-2}} \dif\xe' \dif\ye'         \\
&\le \norm{\omega}_{L^\infty(\De_{\delta})}^s \normSZ{\grad'\varrho}{1-1/s}s{\De_{\delta}}^s             +
\norm{\grad'\varrho}_{L^\infty(\De_{\delta})}^s \normSZ{\omega}{1-1/s}s{\De_{\delta}}^s 
= \textrm{III} + \textrm{IV}.
\end{align*}

Likewise, we estimate \textrm{II} as follows:
\begin{align*}
\textrm{II} \le \normSZ{\varrho}{1-1/s}{s}{\De_{\delta}}^s 
\norm{\grad'\omega}_{L^\infty(\De_{\delta})}^s                                                           +   
\norm{\varrho}_{L^\infty(\De_{\delta})}^s  \normSZ{\grad'\omega}{1-1/s}s{\De_{\delta}} 
&= \textrm{V} + \textrm{VI}.
\end{align*}

\noindent \boxed{4}
Introducing a co-ordinate transformation $\xe' = \delta \hat{\xe}'$ implies 
\begin{align*}
\normSZ{\varrho}{1-1/s}s{\De_{\delta}} 
& = \delta^{\frac{n}{s}-1} \normSZ{\hat\varrho}{1-1/s}s{\De_{1}} ,
\\
\normSZ{\grad'\varrho}{1-1/s}s{\De_{\delta}} 
& = \delta^{\frac{n}{s}-2} \normSZ{\hat{\grad}'\hat\varrho}{1-1/s}s{\De_{1}} , \\
\norm{\grad'\varrho}_{L^\infty(\De_\delta)} &= \delta^{-1} \|\hat{\grad} \hat{\varrho}\|_{L^\infty(\De_1)} , 
\end{align*}
where $\De_1$ denotes the unit disc. 
We now show that \textrm{III}--\textrm{VI} 
are $\mathcal{O}(1)$ in $\delta$. Using \eqref{eq:Linfomega}-\eqref{eq:Semiomega}, we obtain 
\begin{align*}
\textrm{III}
&\lesssim \delta^{\del{2-\frac{n}{s}}s} \delta^{s\del{\frac{n}{s}-2}} 
\normSZ{\omega}{2-1/s}s{\De_\delta} = \normSZ{\omega}{2-1/s}s{\De_\delta} ,  \\
\textrm{IV}
&\lesssim \frac{1}{\delta^s} \delta^n \delta^{s\del{1-\frac{n}{s}}} 
\normSZ{\omega}{2-1/s}s{\De_\delta} = \normSZ{\omega}{2-1/s}s{\De_\delta} , 
\end{align*}
as well as
\[
\textrm{V}   \lesssim \delta^{s\del{\frac{n}{s}-1}}\delta^{s\del{1-\frac{n}{s}}} 
 \normSZ{\omega}{2-1/s}s{\De_\delta} = \normSZ{\omega}{2-1/s}s{\De_\delta}  . 
\]
Since the estimate of \textrm{VI} is immediate, we conclude the proof. 
\end{proof}

Next we extend $\calC\omega$ by zero to the boundary $\bdy \Theta(\xe,
\delta)$ of the bubble domain $\Theta(\xe,\delta)$ but still indicate it
by $\calC\omega$. We denote by $\calE\omega$ the harmonic extension of
$\calC\omega$ to $\Theta(\xe,\delta)$
\cite[Theorem~9.15]{DGilbarg_NTrudinger_2001a}, and point out
that this is not the usual extension which retains the regularity properties of $\calC\omega$.
We collect the properties of $\calE\omega$ in the following lemma.
\begin{lemma}[harmonic extension]
\label{lem:Hext}
Let $\Omega$ be a domain of class $W^{2-1/s}_s$, $\vec x \in \bdy\Omega$, and $\omega$ satisfy \eqref{eq:omega_sobolev_gradient}-\eqref{eq:omega_sobolev_integral}. The harmonic extension $\calE\omega$ of $\calC\omega$ to $\Theta(\xe,\delta)$ satisfies
\begin{subequations}
\begin{align}
\norm{\calE\omega}_{L^\infty(\Theta(\xe, \delta))} 
&\lesssim \norm{\omega}_{L^\infty(\De(\xe', \delta))} , \label{eq:LinfHext}\\
\normS{\calE\omega}2s{\Theta(\xe, \delta)} 
&\lesssim \normS{\omega}{2-1/s}s{\De(\xe', \delta)}  \label{eq:W2sHext} . 
\end{align}
\end{subequations}
\end{lemma}
\begin{proof}
As $\bdy \Theta(\xe, \delta)$ is smooth and $\calC\omega \in \sob{2-1/s}s{\bdy \Theta(\xe, \delta)}$, the harmonic extension $\calE\omega$ exists in $\sob2s{\Theta(\xe,\delta)}$ and
satisfies \eqref{eq:W2sHext} \cite[Lemma~9.17]{DGilbarg_NTrudinger_2001a}. Expression \eqref{eq:LinfHext} is due to the maximum principle.
\end{proof}

We remark that \lemref{lem:Hext} is the first place where we have used the bubble domain and the smoothness of its boundary. Since the map $\Psi$ in \eqref{eq:loc_diff} is defined over $\bbR^n_{-}$, we introduce an extension $\widetilde{\calE}\omega$ of $\calE\omega$ as follows
\begin{equation}\label{eq:tildeE}
\widetilde\calE\omega := \varrho \calE \omega , 
\end{equation}
where $\varrho$ is the cutoff function in
Definition~\ref{def:smooth_char}. In view of \eqref{eq:tildeE} we
remark that $\Psi$ in \eqref{eq:loc_diff} satisfies
(P\ref{P1})-(P\ref{P3}) in \S \ref{s:notation}.
It remains to show that $\Psi$ is a local $W^2_s$--diffeomorphism, but that requires studying the following properties of $\widetilde\calE\omega$.
\begin{lemma}[properties of $\widetilde\calE\omega$]
\label{lem:tcalE}
Let $\Omega$ be a domain of class $W^{2-1/s}_s$, $\vec x \in \bdy\Omega$, and $\omega$ satisfy \eqref{eq:omega_sobolev_gradient}-\eqref{eq:omega_sobolev_integral}. Then
   \begin{subequations}\begin{align}
       \label{eq:omega_estimate_b}
        |\widetilde\calE\omega|_{W^1_\infty(\Theta\of{\xe,\delta})}
            &\lesssim \delta^{1-n/s}\normSZ{\omega}{2-1/s}{s}{\De\of{\xe', \delta}} , \\
       \label{eq:omega_estimate_a}
        \| \widetilde\calE\omega\|_{W^2_s(\Theta\of{\xe, \delta})}  
            &\lesssim \normS{\omega}{2-1/s}{s}{\De\of{\xe', \delta}} .        
    \end{align}\end{subequations}
\end{lemma}
\begin{proof}
For simplicity we use the notation 
$\Theta_{\delta} = \Theta(\xe,\delta)$, $\De_{\delta} = \De(\xe',\delta)$, and $\grad = \grad_\xe$. 

\noindent \boxed{1} The proof of \eqref{eq:omega_estimate_b} is tricky
and we split it into two steps. 
Using the definition of $\widetilde\calE\omega$ we obtain
\[
\|\grad\widetilde\calE\omega\|_{L^\infty(\Theta_{\delta})}
\le \norm{\grad\varrho}_{L^\infty(\Theta_{\delta})} \norm{\calE\omega}_{L^\infty(\Theta_{\delta})}        +
\norm{\varrho}_{L^\infty(\Theta_{\delta})} \norm{\grad\calE\omega}_{L^\infty(\Theta_{\delta})} .
\]
As $\norm{\grad\varrho}_{L^\infty(\Theta_{\delta})} \lesssim \frac{1}{\delta}$ and $\norm{\calE\omega}_{L^\infty(\Theta_{\delta})} \lesssim \delta 
\norm{\grad\calE\omega}_{L^\infty(\Theta_{\delta})}$ by Poincar\'e
inequality, we deduce 
\[
\| \grad\widetilde\calE\omega\|_{L^\infty(\Theta_{\delta})} 
\lesssim \norm{\grad\calE\omega}_{L^\infty(\Theta_{\delta})} . 
\]

\noindent \boxed{2}
Let $\Be = \Be\big((\xe',\xe^n-\frac{\delta}{2}),\frac{\delta}{4}\big)$ be
  the ball of center $(\xe',\xe^n-\frac{\delta}{2})$ and radius $\frac{\delta}{4}$ 
depicted in \figref{f:internal_estimate}.
Adding and subtracting $\overline{\grad\calE\omega} := \fint_{\Be} \grad \calE\omega$ to
$\norm{\grad\calE\omega}_{L^\infty(\Theta_{\delta})}$,
and applying triangle inequality we arrive at
\begin{align*}
\norm{\grad\calE\omega}_{L^\infty(\Theta_\delta)}
&\le \| \grad\calE\omega - \overline{\grad\calE\omega}\|_{L^\infty(\Theta_{\delta})}       +
\|\overline{\grad\calE\omega}\|_{L^\infty(\Theta_{\delta})}     
= \textrm{I} + \textrm{II} . 
\end{align*}

\begin{figure}[h]
    \centering
    \includegraphics[width=0.25\textwidth]{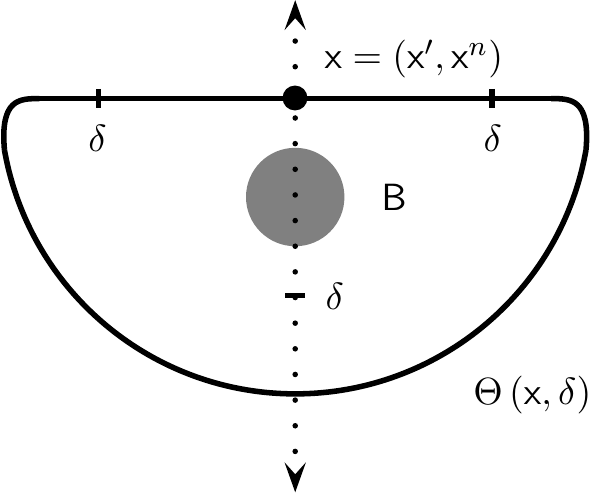}        
    \caption{The bubble domain $\Theta(\xe,\delta)$ centered at $\xe$ (solid curve) and ball $\Be = \Be\big((\xe',\xe^n-\frac{\delta}{2}),\frac{\delta}{4}\big)$.}
    \label{f:internal_estimate}
\end{figure}

We deal with terms \textrm{I} and \textrm{II} separately. To estimate \textrm{I} we use Poincar\'e inequality
\[
\textrm{I} = \|\grad\calE\omega - \overline{\grad\calE\omega}\|_{L^\infty(\Theta_{\delta})}
\lesssim \delta^{1-\frac{n}{s}} \norm{\grad\calE\omega}_{\sob1s{\Theta_{\delta}}}  = \delta^{1-\frac{n}{s}} | \calE\omega |_{W^2_s(\Theta_\delta)} .
\]

On the other hand, we estimate \textrm{II} as follows:
$
\textrm{II} 
\le \fint_\Be \abs{\grad\calE\omega}             
\le \norm{\grad\calE\omega}_{L^\infty(\Be)} . 
$
Invoking the interior estimate for derivatives of a harmonic function \cite[Theorem 2.10]{DGilbarg_NTrudinger_2001a} 
we obtain
\[
\normL{\grad\calE\omega}\infty{\Be}
\leq n\delta^{-1}\normL{\calE\omega}\infty {\Theta_{\delta}} . 
\]
Using \eqref{eq:LinfHext} followed by \eqref{eq:Linfomega} yields 
\[
\textrm{II} \lesssim \delta^{-1} \normL{\omega}\infty{\De_{\delta}} 
\lesssim \delta^{1-\frac{n}{s}} \normSZ{\omega}{2-1/s}s{\De_{\delta}}. 
\]
The estimates for \textrm{I} and \textrm{II} yield $|\calE\omega|_{W^1_\infty(\Theta_\delta)} \le \delta^{1-\frac{n}{s}}  |\omega|_{W^{2-1/s}_s(\De_\delta)}$, whence \eqref{eq:omega_estimate_b} follows.

\noindent \boxed{3} To prove \eqref{eq:omega_estimate_a} we use the definition of 
$\widetilde\calE\omega$ to arrive at 
\begin{multline*}
\|\widetilde\calE\omega\|_{\sob{2}s{\Theta_{\delta}}} 
\lesssim \|D^2\varrho\|_{L^s(\Theta_\delta)} \normL{\calE\omega}\infty{\Theta_\delta}
+ \normL{\grad\varrho}s{\Theta_\delta} \normL{\grad\calE\omega}\infty{\Theta_\delta} \\
+ \|D^2\calE\omega\|_{L^s(\Theta_\delta)}    
+ \normL{\grad\varrho}s{\Theta_\delta} \normL{\calE\omega}\infty{\Theta_\delta} 
+ \normL{\grad\calE\omega}\infty{\Theta_\delta}
+\normL{\calE\omega}s{\Theta_\delta} .
\end{multline*}
Since $\normL{D^2\varrho}s{\Theta_\delta} \lesssim \delta^{\frac{n}{s} - 2}$ 
and
$\normL{\calE\omega}\infty{\Theta_\delta} \lesssim \delta \normL{\grad\calE\omega}\infty{\Theta_\delta}$
(Poincar\'e inequality), Step \boxed{2} shows that the first term is bounded by $\normSZ{\omega}{2-1/s}s{\De_{\delta}}$. 
Similar arguments, in conjunction with \eqref{eq:W2sHext} shows that the remaining terms are also bounded by $\normSZ{\omega}{2-1/s}s{\De_{\delta}}$, 
which is \eqref{eq:omega_estimate_a}.
\end{proof}
\begin{corollary}
[$W^2_s$-diffeomorphism] 
\label{cor:omega_diffeomorphism}
Let $\Omega$ be a domain of class $W^{2-1/s}_s$, $\vec x \in
\bdy\Omega$, and $\omega$ satisfy
\eqref{eq:omega_sobolev_gradient}-\eqref{eq:omega_sobolev_integral}. If
  $\delta>0$ is sufficiently small, then $\Psi$ defined in \eqref{eq:loc_diff} is a local
$W^2_s$--diffeomorphism and satisfies 
\begin{align}
        \label{eq:omega_estimate_c}   
        \normLi{1 - \det{\grad_{\xe}\Psi}}{\Theta\of{\xe, \delta}}
            &\lesssim \delta^{1-n/s}\normSZ{\omega}{2-1/s}s{\De\of{\xe', \delta}} .
\end{align}            
\end{corollary}
\begin{proof}
In view of Definition~\ref{def:loc_diff} and the inverse function theorem, 
$\Psi$ is a local $W^2_s$-diffeomorphism, if and only if $\grad_{\ye}
\Psi(\ye)$ is an isomorphism for every
$\ye \in \Theta(\xe, \delta)$.
We first observe that the definition of $\Psi$ in \eqref{eq:loc_diff} and the properties of $\widetilde\calE\omega$ in \lemref{lem:tcalE} yield $\Psi \in \sob2s{\bbR^n}$. Moreover, 
\eqref{eq:loc_diff} implies
\[
\det{\grad_{\ye}\Psi}(\ye) = 1 - \partial_{\ye^n} \widetilde\calE \omega(\ye)   ,  
\]
whence
\[
\normL{1-\det{\grad_{\ye}\Psi}}\infty{\Theta(\xe,\delta)} 
\le \| \grad_{\ye}\widetilde\calE\omega\|_{L^\infty(\Theta(\xe,\delta))}  .
\]
Using \eqref{eq:omega_estimate_b} we obtain \eqref{eq:omega_estimate_c}. 
Finally, upon choosing $\delta$ small enough in \eqref{eq:omega_estimate_c}, for every $\ye \in \Theta(\xe, \delta)$, we obtain $\abs{1 - \det{\grad_{\ye} \Psi}} < 1/2$, whence $\abs{\det \grad_{\ye} \Psi(\ye)} > 1/2$. Thus, $\grad_{\ye}\Psi (\ye)$ is invertible for every $\ye \in \Theta(\ye,\delta)$, which completes the proof. 
\end{proof}
%


\subsection{Piola Transform}
\newcommand{\matPH}{{\mat{\hat{P}}}}

The purpose of this section is to analyze the Piola transform, a
mapping which preserves the essential boundary condition $\vec u\cdot \vnu = 0$
after the boundary of $\Omega$ has been flattened. We will restrict the presentation to a \emph{local $W^2_s$-diffeomorphism} $\Psi$, $s > n$. 
which maps the reference domain $\Qe$ (bounded or unbounded)
one-to-one and onto a physical domain $Q$, i.e.,  
	\begin{equation*}
		\fullfunction{\Psi}{\Qe}{Q}{\xe}{\vec x},
	\end{equation*}
with $\Ue = \supp\of{\calI - \Psi}$ and $U = \Psi\of{\Ue}$. The
construction of such a $\Psi$ was the subject of
Corollary~\ref{cor:omega_diffeomorphism}, but we do not need in this section to use a particular form of $\Psi$. We remark that for the $\Psi$ constructed in Corollary~\ref{cor:omega_diffeomorphism}, we have $\Ue = \supp\of{\calI - \Psi} \subset \Be_{-}(\xe,\delta)$.

\begin{definition}[Piola transform]
\label{def:piola}
Let $\vec P : \Qe \rightarrow \bbR^{n\times n}$ and $\vec P^{-1} : Q \rightarrow \bbR^{n \times n}$ be the maps $\mat{P} := {J_{\xe}}^{-1}\grad_{\xe} \Psi$, $\matP^{-1} := {J_{\vec x}}^{-1}\grad_{\vec x} \Psi^{-1}$ 
with $J_{\xe} = \det{\grad_{\xe} \Psi}$ and $J_{\vec x} = \det \grad_{\vec x} \Psi^{-1}$. We say two vector fields $\vec v : Q \rightarrow \bbR^n$ and $\ve : \Qe \rightarrow \bbR^n$ are the Piola transforms of each other if and only if 
	\begin{equation}
	\label{eq:bhs_piola_of_v}
	    \vec v \composed \Psi = \matP \ve  , \quad
	    \ve \composed \Psi^{-1} =  \matP^{-1} \vec v ,
	\end{equation}
In view of the inverse function theorem we also have $ \mat{P}^{-1} \composed \Psi = J_{\xe} (\del{\grad_{\xe} \Psi}^{-1} \composed \Psi)$.
\end{definition}

The Piola transform is instrumental in dealing with vector-valued functions because it preserves the divergence and normals 
as the following identities illustrate. 

\begin{lemma}
[Piola identities] 
\label{lem:piola}
If $\ve \in [H^1(\Qe)]^n$ is the Piola transform of $\vec v \in [H^1(Q)]^n$ and $q \in H^1(Q)$, then the following statements hold:
    \begin{subequations}
	\begin{align}
		\label{eq:piola_a}
		 \int_{Q} \grad_{\vec x} q \cdot \vec v \dif\vec x 
		    &=\int_{\Qe} \grad_{\vec\xe}\qe \cdot \vec\ve \dif \vec\xe, \\
		\label{eq:piola_b}    
		\int_{Q} q \divg_{\vec x}{\vec v} \dif\vec x
		    &= \int_{\Qe} \qe \divg_{\vec\xe}{\vec\ve} \dif \vec\xe, \\
		\label{eq:piola_c}    
        \int_{\bdy Q} q \vec v\cdot\vec\nu_{\vec x} \dif s_{\vec x} 
		    &= \int_{\bdy\Qe} \qe \vec\ve \cdot \vnu_{\vec\xe} \dif s_{\vec\xe} , 
	\end{align}
	\end{subequations}
where $\vec\nu_{\vec x}$ and $\vec\nu_{\vec \xe}$ are outward unit normals on $\bdy Q$ and $\bdy\Qe$ respectively. 
Moreover,
\begin{equation}\label{eq:piola_identity}
\vec v\cdot\vec\nu_{\vec x} \dif s_{\vec x} 
= \vec\ve \cdot \vnu_{\vec\xe} \dif s_{\vec\xe} .    
\end{equation}
\end{lemma}
\begin{proof}
The Definition~\ref{def:piola} is precisely what yields \eqref{eq:piola_a}. In fact, 
\[
\int_{Q} \vec v \cdot \grad_{\vec x} q \dif\vec x
= \int_{\Qe} (\vec v \composed \Psi) \cdot \del{(\grad_{\xe}\Psi)^{-T} \composed \Psi} 
  \grad_{\xe} \qe \ J_{\xe} \dif \xe ,
\]
and using $\mat{P}^{-1} \composed \Psi = J_{\xe} (\del{\grad_{\xe} \Psi}^{-1} \composed \Psi)$ 
and Definition~\ref{def:piola}, we deduce 
\[
\int_{\Qe} (\vec v \composed \Psi) \cdot \del{(\grad_{\xe}\Psi)^{-T} \composed \Psi} \grad_\xe \qe J_\xe \dif \xe 
= \int_{\Qe} ((\matP^{-1} \vec v) \composed \Psi ) \cdot \grad_{\xe} \qe \dif \xe 
= \int_{\Qe} \ve \cdot \grad_{\xe} \qe \dif \xe . 
\]
Invoking the divergence theorem and \eqref{eq:piola_a} we obtain 
\begin{align*}
\int_{\bdy Q} q \vec v\cdot\vec\nu_{\vec x} \dif s_{\vec x} 
-\int_{Q} q \divg_{\vec x} \vec v  \dif\vec x  
&= \int_{Q} \vec v \cdot \grad_{\vec x} q \dif\vec x \\
&=\int_{\Qe} \vec\ve \cdot \grad_{\vec\xe} \qe \dif\vec\xe
= - \int_\Qe \qe \divg_{\xe} \ve \dif \xe 
+ \int_{\bdy \Qe} \qe \ve \cdot \vec \nu_{\xe} \dif s_{\xe} 
\end{align*}
Choosing $\qe \in H^1_0(\Qe)$, we obtain \eqref{eq:piola_b}, which 
in turn implies \eqref{eq:piola_c} and \eqref{eq:piola_identity}. 
\end{proof}

\begin{lemma}
[$L^r$-norm equivalence]
\label{lem:LrNorm_equiv}
Let $1 < r < \infty$ and $\ve$ be the Piola transform of $\vec v$. There exists constants $C'$ and $C''$ which depend only on the Lipschitz seminorms of $\Psi$ and $\Psi^{-1}$ such that
    \begin{align}
        \label{eq:bhs_piola_lr_estimate}
            C'\normLr{\ve}{\Qe} 
                \leq \normLr{\vec v}{Q}
                \leq C'' \normLr{\ve}{\Qe}.
    \end{align}
\end{lemma}
\begin{proof} 
This is a trivial consequence of Definition~\ref{def:piola} because 
\[
\normL{\ve}{r}\Qe 
= \normL{(\matP^{-1} \vec v)\composed \Psi}r{\Qe}
\leq \normL{\matP^{-1}\composed \Psi}\infty{\Qe}
\normLi{J_{\vec x}}{Q} \normLr{\vec v}{Q} ,
\]
and
\[
\normL{\vec v}{r}Q 
= \normL{\del{\matP \ve}\composed \Psi^{-1}}rQ
\leq  \normLi{\matP\composed \Psi^{-1}}{Q}  
\normLi{J_{\xe}}{\Qe} \normLr{\ve}{\Qe}.
\]
This concludes the proof.
\end{proof}

\lemref{lem:piola} discusses the transformation of the divergence operator from the reference domain $\Qe$ to the physical domain $Q$. In what follows, we need to transform the gradient operator as well.  

\begin{lemma}
[Piola gradient] 
\label{lem:bhs_piola_grad}
\begin{subequations}
Let $\ve$ and $\vec v$ be Piola transforms of each other. The gradient
operator admits the following decomposition
	\begin{align}
        \label{eq:bhs_piola_grad_b}
        \grad_{\xe}\ve
            &= J_{\xe} \mathfrak{P}_{\matP^{-1}} (\grad_{\vec x} \vec v \composed \Psi)
                + \grad_{\xe} (\matP^{-1} \composed \Psi) \odot \vec v \composed \Psi, \\
		\label{eq:bhs_piola_grad}
		\grad_{\vec x} \vec v \composed \Psi 
		    &= {J_{\xe}}^{-1}  \mathfrak{P}_{\matP}(\grad_{\xe} \ve )  
		        - {J_{\xe}}^{-1} \mathfrak{P}_{\matP} \del{\grad_{\xe}(\matP^{-1} \composed \Psi)\odot\matP \ve }  ,
	\end{align}
where $\mathfrak{P}$ is defined in \eqref{eq:frakP}. Moreover, for $s' \leq t^\bullet \leq s$ 
and $1/t^\bullet = 1/s + 1/t^\circ$ 
    \begin{equation}\begin{aligned}
        \label{eq:bhs_piola_grad_estimate}
            \normSZ{\ve}1{t^\bullet}{{\Qe}} 
                &\leq C'\del{\normSZ{\vec v}1{t^\bullet}{Q} 
                    + \normL{\vec v}{t^\circ}{U}}
            \\
            \normSZ{\vec v}1{t^\bullet}{Q} 
                &\leq C''\del{\normSZ{\ve}1{t^\bullet}{\Qe}
                    + \normL{\ve}{t^\circ}{\Ue}} , 
    \end{aligned}\end{equation}
where the constants $C'$ and $C''$ depend only on $n, r, s$, the Lipschitz and $W^2_s$ semi-norms of $\Psi$ and $\Psi^{-1}$, 
and on the sets $\Ue = \supp\of{\calI - \Psi}$ and $U = \Psi\of{\Ue}$.
\end{subequations}
\end{lemma}
\begin{proof} 
To obtain \eqref{eq:bhs_piola_grad_b} it suffices to differentiate the Piola transform given in \eqref{eq:bhs_piola_of_v} and note that we use the chain rule to deal with the first term $\grad_{\vec x} \vec v$ but not 
the second one $\grad_{\xe}(\matP^{-1} \composed \Psi)$. 

To derive \eqref{eq:bhs_piola_grad} we multiply by $\matP$ on the left and $\matP^{-1}$ on the right of \eqref{eq:bhs_piola_grad_b} and then reorder terms 
\begin{align*}
\matP \del{ \grad_{\xe} \ve 
-  \grad_{\xe} (\matP^{-1} \composed \Psi) \odot \vec v \composed \Psi  }\matP^{-1}
= J_{\xe} \grad_{\vec x} \vec v \composed \Psi ,
\end{align*}
whence, upon using $\vec v \composed \Psi = \matP \ve$ from \eqref{eq:bhs_piola_of_v}, we achieve the desired
expression. 

To show \eqref{eq:bhs_piola_grad_estimate} we observe that $\Psi$ is the identity on 
$\Qe \setminus \Ue$, whence 
$
\normSZ{\ve}1{t^\bullet}{\Qe \setminus \Ue} = \normSZ{\vec v}{1}{t^\bullet}{Q \setminus U} . 
$
To deal with the remaining part on $\Ue$, we resort to \eqref{eq:bhs_piola_grad_b} and $\Psi \in \sob2s\Ue$. 
Combining these results yields the first estimate of \eqref{eq:bhs_piola_grad_estimate}. To obtain the second one, it 
suffices to follow the same steps above starting with \eqref{eq:bhs_piola_grad}.
\end{proof}

\begin{proposition}
[Piola symmetric gradient]
\label{prop:piola_sym_grad_decomposition}
\begin{subequations} \label{eq:piola_sym_grad_decomposition}
Let $\ve$ be the Piola transform of $\vec v$. The symmetric gradient admits the following decomposition 
    \begin{equation}
       \gradS{\vec v}\composed \Psi 
           = {J_{\xe}}^{-1} \del{
            \vec{\varepsilon}_{\matP}(\ve) - \vec \vartheta_{\matP}\of{\ve}}, 
    \end{equation}
with
    \begin{align} \label{eq:hat_varTheta} 
    \begin{aligned} 
        \vec{\varepsilon}_{\matP}(\ve) &:= \mathfrak{P}_{\matP}(\vec\varepsilon(\ve)) , \\ 
         \vec\vartheta_{\matP}\of{\ve}
             &:= \frac{1}{2}\mathfrak{P}_{\vec P} \del{
                 \grad_{\xe}(\matP^{-1} \composed \Psi) \odot \matP \ve 
                 + \del{\grad_{\xe}(\matP^{-1} \composed \Psi) \odot \matP \ve }\transpose } .
    \end{aligned}                 
    \end{align}
Moreover, if $\Psi$ is a local $W^2_s$-diffeomorphism, $s > n$,
and $t^\circ, t^\bullet$ satisfy  
  $s' \leq t^\circ \leq \infty$, $1 \le t^\bullet \le s$, and
  $1/t^\bullet = 1/s + 1/t^\circ$, then there holds 
    \begin{equation}\begin{aligned}
        \label{eq:piola_sym_grad_bound}
          \normL{ \vec{\varepsilon}_{\matP}(\ve) }{t^\circ}{\Qe}
               &\leq C' \normSZ{\ve}1{t^\circ}{\Qe} \\
           \normL{\vec\vartheta_{\matP}\of{\ve}}{t^{\bullet}}{\Qe}
               &\leq C'' \normL{\ve}{t^\circ}{\Ue},
    \end{aligned}\end{equation}
with constants $C'$ and $C''$ depending only on $n, r, s, t^\circ$, the Lipschitz and $W^2_s$ semi-norms of $\Psi$ and $\Psi^{-1}$, and on the sets $\Ue = \supp\of{\calI - \Psi}$ and $U= \Psi\of{\Ue}$.

\end{subequations}
\end{proposition}
\begin{proof} The decomposition follows directly by the definition of the symmetric gradient and \eqref{eq:bhs_piola_grad}. The bounds follow from the bounds in \lemref{lem:bhs_piola_grad}.
\end{proof}

\begin{theorem}
[space isomorphism]
\label{thm:bhs_isomorphic_domains}
Let $s' \leq r \leq s$ and $\Psi$ be a local $W^2_s$-diffeormorphism between $\Qe$ and $Q$. The linear operator 
\begin{equation}\label{eq:cP}
\fullfunction{\calP}{X_{r}\of{\Qe}}{X_{r}\of{Q}}
{\of{\ve, \qe}}{\of{\vec v, q} = \del{\matP \ve, \qe}\composed \Psi^{-1}}
\end{equation}
is an isomorphism.
\end{theorem}
\begin{proof}
We only show that $\calP$ is bounded because the same procedure applies to $\calP^{-1}$.
Consequently,
given $(\ve, \qe) \in X_r(\Qe)$, we will show that $(\vec v,q) \in X_r(Q)$ and the norm 
\[
\| (\vec v, q) \|_{X_r(Q)} 
= \| \calP(\ve , \qe) \|_{X_r(Q)} 
= \| \vec v \|_{V_r(Q)} + \| q \|_{L^r(Q)} 
\]
is bounded. We first observe that $\ve \cdot \vec \nu_{\xe} = 0$ implies $\vec v \cdot \vec\nu_{\vec x} = 0$ because of \eqref{eq:piola_identity}. In view of \lemref{lem:equiv_norm} and \eqref{eq:banach_space_X_snorm}, $V_r(Q)$ is complete under the semi-norm $\normSZ{\cdot}1r{Q}$. Owing to \eqref{eq:bhs_piola_grad_estimate}, for $r = t^{\bullet}$, we obtain
\begin{align*}
 \normSZ{\vec v}1{r}{Q} 
\lesssim \normSZ{\ve}1{r}{\Qe} + \normL{\ve}{t^\circ}{\Ue} ,
\end{align*}
and see that the first term is bounded because $\ve \in V_r(\Qe)$. To bound the second term, we use the Sobolev embedding theorem $\sob1r\Ue \subset L^{t^\circ}(\Ue)$ with
$1/t^\circ = 1/r - 1/s$ to arrive at  $\normL{\ve}{t^\circ}{\Ue} \lesssim \normS{\ve}1r\Ue$. 

If $\Qe$ is bounded then \lemref{lem:equiv_norm} implies
$
\normS{\ve}1r\Ue \lesssim \normS{\vec v}1r\Qe 
\lesssim \normSZ{\ve}1r\Qe . 
$ 
If $Q = \bbR^n$, then $\calP = \calI$ and there is nothing to prove. 
If $Q = \bbR^n_{-}$ then by Remark~\ref{rem:norm_equiv} we obtain 
$
\normS{\ve}1r\Ue \le \normSZ{\ve}1r\Ue .
$ 
Altogether, we conclude that $\vec v \in V_r(Q)$. 

It remains to estimate $\norm{q}_{L^r(Q)}$. Due to the change of
variables and the fact that $\Psi$ is $W^2_s$ with $s > n$,
we arrive at $\norm{q}_{L^r(Q)} \lesssim \normLi{J_{\xe}}{\Qe}\normLr{\qe}{\Qe}$. 
This concludes the proof. 
\end{proof}

\begin{remark}[$W^2_s$--regularity] \rm
The proof of \thmref{thm:bhs_isomorphic_domains} reveals that it is absolutely necessary for $\Psi$ to have two 
derivatives, i.e. $\Psi \in W^2_s$, for the Piola transform to make sense as an isomorphism between $X_{r}\of{\Qe}$ and $X_{r}\of{Q}$. This 
differs from the canonical use of the Piola transform for $H\of{\divg}$ spaces  which hinges on 
\eqref{eq:piola_identity}. 
\end{remark}


\section{The Sobolev Space Case}
\label{s:sobolev_case}
We start with a brief summary of \emph{index theory} and related results following \cite{PDLax_2002}. 
Let $X$, $Y$ and $Z$ be arbitrary Banach spaces with $\dual{X}$, $\dual{Y}$ and $\dual{Z}$ being their duals.

A (bounded) linear operator $\calA : X \to Y$ is said to have \emph{finite index} if it has the following properties:
    \begin{enumerate}[(i)]
        \item The nullspace $N_\calA$ of $\calA$ is a finite dimensional subspace of $X$.
        \item The quotient space $Y/R_\calA$ is finite dimensional, with $R_\calA$ the range of $\calA$.
    \end{enumerate}
For such an operator we define the \emph{index} as 
    \[
        \ind \calA := \dim N_\calA - \dim Y/R_\calA.
    \]
Two bounded linear operators $\calA : X \to Y$ and $\calA^\dagger: Y \to X$ are called \emph{pseudoinverses} of each other if 
    \[
        \calA \calA^\dagger = \calI_Y + \calK, \quad \calA^\dagger \calA = \calI_X + \calC,
    \]
where $\calK : Y \rightarrow Y$ and $\calC : X \rightarrow X$ are \emph{compact} operators. 
Every bounded linear operator $\calA : X \to Y$ has a dual (operator) $\dual\calA : \dual Y \to \dual X$ given by the relation, 
    \[
        \pair{\dual\calA y^*, x}_{\dual X, X} := \pair{y^*, \calA x}_{\dual Y, Y}, \quad  x \in X, \, y^* \in \dual Y.
    \]

\begin{lemma} 
[index vs pseudoinverse {\cite[Chapter 27: Theorems 1,2]{PDLax_2002}}]
\label{lem:lax_c}
    A bounded linear operator $\calA : X \to Y$ has finite index if and only if $\calA$ has a pseudoinverse. Moreover, 
    \[
        \ind A = -\ind A^\dagger.
    \]
\end{lemma}


\begin{lemma}[compact perturbation {\cite[Chapter 21: Theorem 3]{PDLax_2002}}]
\label{lem:lax_d}
    Suppose that $\calA : X \to Y$ has finite index, and $\calK: X \to Y$ is a compact linear map. Then $\calA + \calK$ has finite index and 
    \[
        \ind\of{\calA + \calK} = \ind \calA .
    \]
\end{lemma}

\begin{lemma}[index of dual {\cite[Chapter 27: Theorem 4]{PDLax_2002}}]
\label{lem:lax_b} 
    Let $\calA: X \to Y$ be a bounded linear operator.
    If $\calA$ has finite index, then so does its dual $\dual{\calA}$. Moreover, 
    \[
        \ind \dual{\calA} = -\ind \calA
        .
    \]
\end{lemma}

%

\begin{corollary}[invertibility]
\label{cor:h_r_s} 
Let $\calA: X \to Y$ be a bounded operator with a pseudoinverse. If $\calA$ and $\dual\calA$ are injective then they are bijective.
\end{corollary}
\begin{proof} From \lemref{lem:lax_c} we have that $\calA$ has finite index. Since $\calA$ and $\dual\calA$ are injective, $\dim N_\calA = \dim N_{\dual\calA} = 0$. According to \lemref{lem:lax_b} we have, $ - \dim \dual X / R_{\dual\calA} = \dim Y/ R_{\calA}$. Since the dimension of a space is not negative, we obtain
    \[ 
        \dim \dual X / R_{\dual\calA} = \dim Y/R_{\calA} = 0,
    \] 
i.e. $\calA$ and $\dual\calA$ are surjective which concludes our proof.
\end{proof}

Our strategy is to use  Corollary~\ref{cor:h_r_s} to infer the invertibility of the Stokes operator $\calS_\Omega: X_{r}\of\Omega \to \dual{ X_{r'}\of\Omega}$. 
\begin{itemize}
\item First, we will decompose $\calS_\Omega$ into its interior and boundary parts (see \secref{s:loceq}). 

\item Second, we will use the boundedness of $\Omega$ to construct a pseudoinverse of $\calS_\Omega$, hence showing that it has a finite index (see \secref{s:decomposition}). 

\item Third, and last, we will show that $\calS_\Omega$ and $\dual\calS_\Omega$ are injective (see \secref{s:stokes_slip_injective}).
\end{itemize}
%

\subsection{Localized Equations}
\label{s:loceq}
The goal of this section is to localize the Stokes equations. The technique's essence is to test the Stokes variational system with a cutoff version of a velocity-pressure pair $\of{\vec v, q}$ defined over an unbounded domain. This exhibits the local behavior, in operator terms, of the Stokes linear map which splits (locally) into bounded operators including invertible $\calS_{\bbR^n}$ or $\calS_{\halfspace}$ plus a compact part.

\begin{definition}[localization operator] 
\label{def:loc_op}
Let $\Omega$ be a $W^{2-1/s}_s$-domain. Let $\vec x \in \clos\Omega$ and 
$\zeta \in \CCinf{B\of{\vec x, \delta}}$. Then for every $s' \leq r \leq s$, 
    \begin{align*}
        &\fullfunction{\calR_{\zeta}}{X_{r}\of{\Qe}}{X_{r}\of{\Omega}}
                     {\of{\ve, \qe}}{\zeta \calP \of{\ve, \qe}}     \\   
        &\fullfunction{\calE_{\zeta}}{X_{r}\of{\Omega}}{X_{r}\of{\Qe}}
                     {\of{\vec v, q}}{\calP^{-1}\of{\zeta \vec v, \zeta q}}         
    \end{align*}
are \emph{localization operators} if and only if
    \begin{itemize}
        \item when $\vec x \in \Omega$, then $\delta = \dist\of{\vec x, \bdy\Omega}$, and $\Psi = \calI$, $\calP = \calI$ and 
        $Q = \Qe = \bbR^n$; 
        \item when $\vec x \in \bdy\Omega$, then $\delta>0$ is
          sufficiently small so that Corollary~\ref{cor:omega_diffeomorphism} holds, $\Ue(\xe,\delta) = \supp\del{\calI - \Psi}\subset \Be_{-}(\xe,\delta)$,
        $U(\vec x,\delta) = \Psi(\Ue)$ (see \figref{f:RefMap}), 
          and $\Qe = \halfspace$, $Q = \Psi\of{\Qe}$. 
    \end{itemize}
\end{definition}


\begin{lemma}[continuity]
\label{lem:cont_RE}
The operators $\calR_{\zeta}$ and $\calE_{\zeta}$ are continuous. 
\end{lemma}
\begin{proof} 
Since $\zeta$ is smooth, this follows from $\calP$ and $\calP^{-1}$ being
continuous. This is due to \thmref{thm:bhs_isomorphic_domains} if
$\vec x \in \bdy\Omega$ and to $\calP = \calI$ if $\vec x \in \Omega$.
\end{proof}

Next we state the equations satisfied by the localized Stokes operator using $\calR_\zeta$ and $\calE_\zeta$. 
This process is applied both to the interior of $\Omega$ in Proposition~\ref{prop:sslip_interior_localization} and to its boundary in Proposition~\ref{prop:sslip_boundary_localization}.

\begin{proposition}[interior localization]
\label{prop:sslip_interior_localization}
 \begin{subequations}
Let $\vec x \in \Omega$. The Stokes (interior) operators    
$\dual{\calR_{\zeta}}\calS_\Omega : X_r\of{\Omega} \to \dual{X_{r'}\of{\bbR^n}}$ and 
$S_\Omega{\calR_{\zeta}} :X_r\of{\bbR^n} \to \dual{X_{r'}\of{\Omega}}$ are linear continuous 
 and can be written as 
    \begin{equation}\begin{aligned}
        \label{eq:stokes_interior_decomposition_short}
            \dual{\calR_{\zeta}}\calS_\Omega 
                &= \widetilde\calS \calE_{\zeta} + \dual{\calP}\calK_{\zeta} \ ,  \quad 
            \calS_\Omega{\calR_{\zeta}} 
                 = \dual\calE_{\zeta} \widetilde\calS + \calK_{\zeta} {\calP} ,
    \end{aligned}\end{equation}
where $\widetilde\calS = \calS_{\bbR^n}$, $\calP : X_r(\bbR^n)
\rightarrow X_r(\bbR^n)$ is the identity, 
$\calK_{\zeta}:X_r\of{\bbR^n}\to \dual{X_{r'}\of{\bbR^n}}$ is given by 
    \begin{equation}\begin{aligned}\label{eq:stokes_localized_fi}
        \calK_{\zeta}\of{\vec u, p}\of{\vec v, q}   
            &:=  
                - \pair{p, \grad_{\vec x} \zeta \cdot \vec v}_{U(\vec x,\delta)}
                - \pair{\grad_{\vec x} \zeta \cdot \vec u, q}_{U(\vec x,\delta)}     \\
           &\quad - \eta \pair{\vec{\mathfrak{e}}_\zeta\of{\vec u}, \gradS{\vec v}}_{U(\vec x,\delta)} 
                + \eta \pair{\gradS {\vec u}, \vec{\mathfrak{e}}_\zeta\of{\vec v}}_{U(\vec x,\delta)},
    \end{aligned}\end{equation}
and 
    \begin{equation}\begin{gathered}
        \label{eq:stokes_interio_theta_estimate}
       \vec{\mathfrak{e}}_\zeta\of{\vec w} 
           := \frac{1}{2}\of{\grad_{\vec x} \zeta \otimes \vec w + \vec w \otimes \grad_{\vec x}\zeta},  \\   
        \normL{\vec{\mathfrak{e}}_\zeta\of{\vec w}}t{U(\vec x,\delta)} 
            \lesssim \normL{\vec w}t{U(\vec x,\delta)},
    \end{gathered}\end{equation}
where $1\leq t \leq \infty$.
\end{subequations}
\end{proposition}
\begin{proof} 
Let $(\vec u, p) \in X_r(\Omega)$, $(\ve, \qe) \in X_r(\bbR^n)$
be fixed. To localize $\calS_\Omega$, we employ a test pair of the form
$\zeta(\ve, \qe)$ and switch the cut-off function $\zeta$ as a
multiplier of the solution pair $(\vec u, p)$.
To do so, we first realize that since $\calR_\zeta^* \calS_\Omega$ and $\calS_\Omega \calR_\zeta$ are compositions
of linear and continuous operators, they are themselves linear and continuous. 
To prove \eqref{eq:stokes_interior_decomposition_short} we recall Definition~\ref{def:loc_op}
\[
\calR_\zeta (\ve, \qe) = \zeta \calP(\ve, \qe) = \zeta (\ve, \qe) ,
\]
because $\calP = \calI$. We multiply by $\calS_\Omega(\vec u,p)$ and rearrange terms to deduce 
\[
\pair{\dual{\calR_{\zeta}}\calS_\Omega\of{\vec u, p}, \of{\ve, \qe}}_{\dual{X_{r'}\of{\bbR^n}}, X_{r'}\of{\bbR^n}}
= \calS_\Omega(\vec u,p)(\zeta \ve, \zeta \qe) . 
\]
We now move the cutoff function $\zeta$ from $(\ve, \qe)$ to $(\vec u,p)$ to obtain 
\[
\calS_\Omega(\vec u,p)(\zeta \ve, \zeta \qe) 
= \calS_\Omega (\zeta\vec u, \zeta p)(\ve, \qe) 
+ \calK_\zeta(\vec u, p) \calP (\ve, \qe) .
\]
Since $\calP$ is the identity, we further have $(\zeta\vec u, \zeta p) = \calE_\zeta(\vec u, p)$ 
thereby getting the first expression of \eqref{eq:stokes_interior_decomposition_short}. 
The proof of the second expression is similar and thus omitted for brevity. 
Using H\"older's inequality, and smoothness of $\zeta$, we get the estimate 
\eqref{eq:stokes_interio_theta_estimate}. This concludes the proof. 
\end{proof}

For simplicity for the rest of this section we let $\eta = 1$.

\begin{proposition}
[boundary localization]
\label{prop:sslip_boundary_localization}
\begin{subequations}
Let $\vec x \in \bdy\Omega$, $Q = \Psi(\bbR^n_{-})$, $\Ue(\xe,\delta)
= \supp\of{\calI - \Psi} \subset \Be_{-}(\xe,\delta)$, and $U(\vec x, \delta) = \Psi(\Ue)$. 
The Stokes (boundary) operators $\dual{\calR_{\zeta}}\calS_\Omega : X_r\of{\Omega} \to \dual{X_{r'}\of{\bbR^n_{-}}}$  
and $S_\Omega{\calR_{\zeta}} :X_r\of{\bbR^n_{-}} \to \dual{ X_{r'}\of{\Omega}}$ 
are continuous and can be written as 
     \begin{equation}\begin{aligned}
        \label{eq:stokes_boundary_decomposition_short}
           \dual{\calR_{\zeta}}\calS_\Omega 
               &= (\widetilde\calS + \calC) \calE_{\zeta} + \dual{\calP}\calK_{\zeta} \ ,  \quad 
            S_\Omega{\calR_{\zeta}} 
               = \dual\calE_{\zeta} (\widetilde\calS + \calC)+ \calK_{\zeta} {\calP}      ,     
    \end{aligned}\end{equation}
where $\widetilde\calS := \calS_\Qe + \calB$ with 
    \begin{align}\label{eq:s_decomp_bi}
    \calB\of{\we, \pi}\of{\ve, \qe} 
        &:=                
            \langle \vec{\varepsilon}_{\matP}\of{\we}, J_{\xe}^{-1}
                   \vec{\varepsilon}_{\matP}\of{\ve}\rangle_{\halfspace} 
            -\pair{\gradS{\we}, \gradS{\ve}}_{\halfspace},
    \end{align}
and
    \begin{equation}
      \label{eq:s_decomp_ci}
    \begin{aligned}  
     \calC\of{\we, \pi}\of{\ve, \qe}
            &:= \langle \vec\vartheta_{\matP}\of{\we}, J_{\xe}^{-1}
                      \vec\vartheta_{\matP}\of{\ve}\rangle_{\Ue(\xe, \delta)} 
            \\ &\quad -
             \langle \vec\vartheta_{\matP}\of{\we},  J_{\xe}^{-1}
                     \vec{\varepsilon}_{\matP}\of{\ve}\rangle_{\Ue(\xe,\delta)}
            - 
            \langle \vec{\varepsilon}_{\matP}\of{\we}, J_{\xe}^{-1}
                      \vec\vartheta_{\matP}\of{\ve}\rangle_{\Ue(\xe,\delta)} .                    
    \end{aligned}
    \end{equation}
The operators $\vec{\varepsilon}_{\matP}$ and $\vec\vartheta_{\matP}$ are defined in 
\eqref{eq:hat_varTheta} and $K_{\zeta}$ is defined in \eqref{eq:stokes_localized_fi}.  
\end{subequations}
\end{proposition}
\begin{proof} 
Let $\of{\vec u, p} \in X_r\of{\Omega}$ and $\of{\ve, \qe} \in
    X_{r'}\of{\bbR^n_{-}}$ be fixed.
As in Proposition~\ref{prop:sslip_interior_localization},
we again take $\zeta\of{\ve, \qe}$ as a test function and switch the
cut-off function $\zeta$ as a mutiplier for the solution $\of{\vec u, p}$.
We obtain
\[
\calS_{\Omega}\of{\vec u, p}\calR_{\zeta}\of{\ve, \qe} 
= \calS_{U(\vec x,\delta)}\of{\zeta \vec u, \zeta p} \calP \of{\ve, \qe}
+ \calK_{\zeta}\of{\vec u, p} \calP \of{\ve, \qe} . 
\]
except that now $\calP$ is no longer the identity. Let $(\vec v, q) = \calP(\ve, \qe)$ and 
$(\zeta \ue, \zeta \pe) = \calP^{-1}(\zeta\vec u, \zeta p)$. The divergence term in $\calS_{U(\vec x, \delta)}$ can be simplified after using 
the Piola identity \eqref{eq:piola_b}
\begin{align*}
\int_{U(\vec x,\delta)} \zeta p \divg_{\vec x} \vec v   \dif \vec x          
= \int_{\Ue(\xe,\delta)} \zeta \pe \divg_{\xe} \ve \dif \xe .
\end{align*}
Moreover for the symmetric gradient we resort to Proposition~\ref{prop:piola_sym_grad_decomposition} to write 
\[
\int_{U(\vec x,\delta)} \vec\varepsilon(\zeta\vec u) : \vec\varepsilon(\vec v) \dif \vec x  
= \int_{\Ue(\xe,\delta)} \del{\vec\varepsilon_{\matP}(\zeta \ue) - \vec\vartheta_{\matP}(\zeta\ue) } : 
\del{\vec\varepsilon_{\matP}(\ve) - \vec\vartheta_{\matP}(\ve) } J_{\xe}^{-1} \dif \xe . 
\]
We add $\pair{\vec\varepsilon(\zeta\ue),\vec\varepsilon(\ve)}_{\bbR^n_{-}}$
to create the term $\calS_{\mathbb{R}^n_-}(\zeta\ue,\pe)(\ve,\qe)$ and subtract it to compensate. The latter, together with the preceding terms give rise to $\calB$ in \eqref{eq:s_decomp_bi} and $\calC$ in \eqref{eq:s_decomp_ci}. The expression for $\dual{\calR_{\zeta}}\calS_\Omega$ follows analogously.
\end{proof}

\begin{lemma}
[$\calK_{\zeta}$ is compact] 
\label{lem:k_compact}
Let $s' \leq r \leq s$, $\vec x \in \clos\Omega$ and $Q = \Psi(\Qe)$, 
where $\Qe = \bbR^n$ $(\mbox{if } \vec x \in \Omega)$, $\Qe = \bbR^n_{-}$, $(\mbox{if } \vec x \in \bdy\Omega)$.
The operator $\calK_{\zeta} : X_r\of{Q} \to \dual{X_{r'}\of{Q}}$ is compact.
\end{lemma}
\begin{proof} 
Let $\{ (\vec u_{\ell}, p_{\ell})\}_{\ell \in \bbN} \subset X_r(Q)$ 
be a bounded sequence. Since $X_r(Q)$ is reflexive, there exists a subsequence 
$\{ (\vec u_{\ell} , p_{\ell}) \}_{\ell \in \bbN}$ (not relabeled) such that 
\[
(\vec u_{\ell}, p_{\ell}) \rightharpoonup (\vec u, p) \quad \mbox{in } X_r(\Omega) ,
\]
and due to Rellich-Kondrachov theorem (cf. \cite[Theorem 6.2]{RAAdams_JJFFournier_2003})
\[
\grad \vec u_{\ell} \rightarrow \grad \vec u \quad \mbox{in } \sob1{r'}\Omega^* , \quad  
p_{\ell} \rightarrow p \quad \mbox{in } \sob1{r'}\Omega^* . 
\] 
For simplicity we denote 
$U = U(\vec x,\delta)$. 
We need to prove 
\begin{equation}\label{lem:k_compact_a}
\calK_\zeta(\vec u_\ell,p_\ell) \rightarrow \calK_\zeta(\vec u,p) \quad \mbox{in } X_{r'}(Q)^*  .  
\end{equation}
We will proceed in several steps. We write \eqref{eq:stokes_localized_fi} as follows
\begin{align*}
\calK_{\zeta}\of{\vec u_\ell, p_\ell}\of{\vec v, q} 
&= \textrm{I} + \textrm{II} + \textrm{III} + \textrm{IV} .   
\end{align*}
In view of \lemref{lem:equiv_norm} we estimate $\textrm{I}$ as
\begin{align*}
|\textrm{I}| 
& \lesssim \norm{p_\ell}_{\sob1{r'}{U}^*} 
    \normS{\grad_{\vec x} \zeta \cdot \vec v}{1}{r'}{U} 
\\
&\lesssim \norm{p_\ell}_{\sob1{r'}{\Omega}^*} \normS{\vec v}{1}{r'}{\Omega}
\lesssim \norm{p_\ell}_{\sob1{r'}{\Omega}^*} \normSZ{\vec v}{1}{r'}{\Omega} 
\lesssim \norm{p_\ell}_{\sob1{r'}{\Omega}^*} \normSZ{\vec v}{1}{r'}{Q} .
\end{align*}
For \textrm{II} and \textrm{III} we use the definition \eqref{eq:stokes_interio_theta_estimate} of $\mathfrak{e}_\zeta$ to obtain 
\[
|\textrm{II}| , |\textrm{III}| 
\lesssim \normL{\vec u_\ell}{r}{\Omega} \Big(\normL{q}{r'}{\Omega} + \normS{\vec v}1{r'}{\Omega}\Big) 
\lesssim \normL{\vec u_\ell}{r}{\Omega} 
\Big(\normL{q}{r'}{Q} + \normSZ{\vec v}{1}{r'}{Q} \Big)  . 
\]
\lemref{lem:equiv_norm} again implies 
\begin{align*}
|\textrm{IV}| 
& \lesssim \norm{\gradS {\vec u_\ell}}_{\sob1{r'}{U}^*} 
    \normS{\vec v}{1}{r'}{U} 
\\
&\lesssim     
\norm{\gradS {\vec u_\ell}}_{\sob1{r'}{\Omega}^*} 
    \normS{\vec v}{1}{r'}{\Omega}  
\lesssim \norm{\gradS {\vec u_\ell}}_{\sob1{r'}{\Omega}^*} 
\normSZ{\vec v}{1}{r'}{Q}    . 
\end{align*}
Collecting estimates, we obtain for all $(\vec v,q) \in X_{r'}(Q)$
\begin{multline*}
\sup_{(\vec v, q) \in X_{r'}(Q)} 
\frac{\abs{\calK_{\zeta}\of{\vec u_\ell, p_\ell}\of{\vec \ve, q} - \calK_{\zeta}\of{\vec u, p}\of{\vec \ve, q} }}{\norm{(\vec v,q)}_{X_{r'}(Q)}} \\
\lesssim \norm{p_\ell - p}_{\sob1{r'}{\Omega}^*}
+ \normL{\vec u_\ell - \vec u}{r}{\Omega}
+ \norm{\gradS {\vec u_\ell} - \gradS {\vec u}}_{\sob1{r'}{\Omega}^*}
\end{multline*}
which tends to $0$ as $\ell\to\infty$
and implies \eqref{lem:k_compact_a}. This concludes the proof.
\end{proof}


\begin{lemma}[$\calC$ is compact]
\label{lem:C_compact}
Let $s' \leq r \leq s$ and $\vec x \in \bdy\Omega$. The operator 
$\calC: X_{r}\of{\bbR^{n}_{-}} \to \dual{X_{r'}\of{\bbR^{n}_{-}}}$ defined in \eqref{eq:s_decomp_ci} is compact.
\end{lemma}
\begin{proof}
The starting point of the proof is the same as in \lemref{lem:k_compact}. For simplicity we set $U = U(\vec x, \delta)$ and $\Ue = \Ue(\xe, \delta)$. 
Let $\{ (\ue_{\ell}, \pe_{\ell})\}_{\ell \in \bbN} \subset X_r(\bbR^{n}_{-})$ be a bounded sequence. 
Since $X_r$ is reflexive, there exists a subsequence 
$\{ (\ue_{\ell} , \pe_{\ell}) \}_{\ell \in \bbN}$ (not relabeled), such that 
\[
\ue_{\ell} \rightharpoonup \ue \quad \mbox{in } V_r(\bbR^n_{-}) , \quad  
\pe_{\ell} \rightharpoonup \pe \quad \mbox{in } L^r(\bbR^{n}_{-}) .
\]
Setting 
\[
\frac{1}{r} = \frac{1}{s} + \frac{1}{r^\circ}   , \quad    \frac{1}{r'} = \frac{1}{s} + \frac{1}{(r')^{\circ}} ,
\]
the following embeddings are compact 
\[
\sob1r\Ue \subset L^{r^\circ}(\Ue) , \quad \sob1{r'}\Ue \subset L^{(r')^\circ}(\Ue) . 
\]
  We rewrite $\calC\of{\ue_\ell, \pe_\ell}\of{\ve, \qe} =
    \textrm{I} + \textrm{II} + \textrm{III}$ with 
\begin{gather*}
\textrm{I}   = \langle \vec\vartheta_{\matP}\of{\ue_\ell}, J_{\xe}^{-1}
                      \vec\vartheta_{\matP}\of{\ve}\rangle_{\Ue}, \quad
\textrm{II}  =                       
             -
             \langle \vec\vartheta_{\matP}\of{\ue_\ell},  J_{\xe}^{-1}
                   \vec{\varepsilon}_{\matP}\of{\ve}\rangle_{\Ue}, \\
\textrm{III} =                      
            - 
            \langle \vec{\varepsilon}_{\matP}\of{\ue_\ell}, J_{\xe}^{-1}
                      \vec\vartheta_{\matP}\of{\ve}\rangle_{\Ue} , 
\end{gather*}
and estimate each term separately. 
The estimate for \textrm{I} reads
\[
\abs{\textrm{I}}
\lesssim \normL{\vec\vartheta_{\matP}\of{\ue_\ell}}{r}\Ue    
    \normL{\vec\vartheta_{\matP}\of{\ve}}{r'}\Ue 
\lesssim \normL{\ue_\ell}{r^\circ}\Ue \normL{\ve}{(r')^\circ}\Ue      
\lesssim \normL{\ue_\ell}{r^\circ}\Ue  \normSZ{\ve}1r{\bbR^n_{-}} , 
\]
in view of \eqref{eq:piola_sym_grad_bound} as well as 
$\normL{\ve}{(r')^\circ}\Ue \le \normS{\ve}1r\Ue \lesssim \normSZ{\ve}1r\Ue$, the latter being a consequence of Remark~\ref{rem:norm_equiv}. 

Applying H\"older's inequality and using $\Psi \in \sob2s{\bbR^n}$ in conjunction with
 \eqref{eq:piola_sym_grad_bound}, we obtain 
\[
\abs{\textrm{II}}
\lesssim 
\normL{\vec\vartheta_{\matP}\of{\ue_\ell}}{r}\Ue  
\normL{\vec\varepsilon_{\vec P}(\ve)}{r'}\Ue 
\lesssim \normL{\ue_\ell}{r^\circ}\Ue \normSZ{\ve}1{r'}{\bbR^n_{-}} . 
\]

Instead of directly estimating \textrm{III}, for every $(\ve, \qe) \in X_r(\bbR^n_{-})$, we consider 
\[
\abs{\langle \vec{\varepsilon}_{\matP}\of{\ue_\ell - \ue}, J_{\xe}^{-1}
                      \vec\vartheta_{\matP}\of{\ve}\rangle_{\Ue}} ,
\]
where we have used the linearity of $\vec{\varepsilon}_{\matP}(\cdot)$. As 
$\vec{\varepsilon}_{\matP}(\cdot) : \sob1r\Ue \rightarrow L^r(\Ue)$ is continuous and 
$J_{\xe}^{-1} \vec\vartheta_{\matP}\of{\ve} \in L^{r'}(\Ue)$, we deduce that 
\[
\lim_{\ell \rightarrow \infty} \abs{\langle \vec{\varepsilon}_{\matP}\of{\ue_\ell - \ue}, J_{\xe}^{-1}
                      \vec\vartheta_{\matP}\of{\ve}\rangle_{\Ue}}  =0 .
\]

Combining the estimates for \textrm{I}, \textrm{II}, \textrm{III}, we obtain 
\[
\lim_{\ell \rightarrow \infty}\abs{\calC\of{\ue_\ell, \pe_\ell}\of{\ve, \qe} - \calC\of{\ue, \pe}\of{\ve, \qe}}
=0 , 
\]
for every $(\ve, \qe) \in X_r(\bbR^{n}_{-})$. This completes the proof. 
\end{proof}

Now that we have obtained a local decomposition of $\calS_\Omega$ it is important to show that the perturbed Stokes operator $\widetilde \calS$ in Propositions~\ref{prop:sslip_interior_localization} and \ref{prop:sslip_boundary_localization} is invertible and enjoys the same smoothing property as $\calS_{\bbR^n}$ and $\calS_{\halfspace}$. The strategy is to use \emph{Neumann} perturbation theorem \cite[Chapter 4: Theorem 1.16]{TKato_1966}, \cite[Lemma 3.1]{GPGaldi_CGSimader_HSohr_1994} which we restate in a form that suits our needs.

\begin{lemma}
[perturbation of identity]
\label{lem:Kato} 
Consider two Banach spaces $X$ and $Y$, and two bounded linear operators $\calA$ and $\calB$ from $X$ to $Y$. Suppose $\calA$ has a bounded inverse from $Y$ to $X$ and that 
	\[
		\norm{\calB x}_Y \leq C \norm{\calA x}_Y \quad \forall x \in X,
	\]
with a constant $0 < C < 1$. Then $\calA+\calB:X \to Y$ is bijective with a bounded inverse.
\end{lemma}


\begin{theorem}
[well-posedness of $\widetilde\calS$]
\label{thm:hs_perturbed_wellposed}
Let $s' \leq r \leq s$ and $\vec x \in \clos\Omega$. There exists a constant $C = C\of{n,r,s,\bdy\Omega}$ such that if $\delta \le C$, then the (perturbed) Stokes problem
\[
    \widetilde\calS \of{\we , \pi} = \calF
\] 
is well-posed from $X_r\of{\Qe}$ to $\dual{X_{r'}\of{\Qe}}$, where $\Qe = \bbR^n$ if $\vec x \in \Omega$ or $\Qe = \bbR^n_{-}$ if $\vec x \in \bdy\Omega$. Additionally, if $r < t \leq s$ and $\calF \in \dual{X_{t'}\of{\Qe}}$, then $\of{\we, \pi} \in X_{t}\of{\Qe}$.
\end{theorem}
\begin{proof} 
For simplicity we set $U = U(\vec x, \delta)$ and $\Ue = \Ue(\xe,
\delta)$.We first note that
if $\vec x \in \Omega$, then $\widetilde\calS = \calS_{\bbR^n}$ from Proposition~\ref{prop:sslip_interior_localization} and $\calS_{\bbR^n}$ is 
invertible thanks to \thmref{thm:rn_wellposed}.

On the other hand,
if $\vec x \in \bdy\Omega$, then $\widetilde\calS = \calS_{\bbR^n_{-}}
+ \calB$ from Proposition~\ref{prop:sslip_boundary_localization} and
$\calS_{\bbR^n_{-}}$ is invertible owing to \thmref{thm:hs_wellposed}.
To prove the invertibility of $\widetilde{\calS}$ using \lemref{lem:Kato}, we will show that for $\delta$ sufficiently small 
\[
\norm{\calB} \le C(\delta) \| \calS_{\bbR^n_{-}} \| ,  \quad C(\delta) < 1 .  
\]
We start by splitting $\calB\of{\we, \pi}\of{\ve, \qe}$ in \eqref{eq:s_decomp_bi} into three parts 
\begin{gather*}
\textrm{I}
= \big\langle\vec{\varepsilon}_{\matP}\of{\we}, ({J_{\xe}}^{-1}-1)
    \vec{\varepsilon}_{\matP}\of{\ve} \big\rangle_{\halfspace} , \quad
\textrm{II}
= \big\langle \vec{\varepsilon}_{\matP}\of{\we}, 
    \vec{\varepsilon}_{\matP}\of{\ve} - \vec\varepsilon(\ve) \big\rangle_{\halfspace} ,  \\
\textrm{III} = \pair{\vec{\varepsilon}_{\matP}\of{\we} 
    -\gradS{\we}, \gradS{\ve}}_{\halfspace} .
\end{gather*}
We recall that $\vec\varepsilon_{\matP}(\we)$ is defined in \eqref{eq:hat_varTheta} and
assume for the remainder of this proof that $\normSZ{\ve}1{r'}\halfspace = 1$. 
We can readily estimate \textrm{I} using \eqref{eq:piola_sym_grad_bound} and \eqref{eq:omega_estimate_c}
\begin{align*}
\abs{\textrm{I}}        
&\leq C \normLi{1 - \det{\grad_{\xe}\Psi}}{\Ue}\normSZ{\we}1r{\Ue} 
\leq C \delta^{1-n/s}\normSZ{\we}1r{\Ue}.    
\end{align*}
The constant $C$ depends on $n$, $r$, $s$, $\Ue$, $\normS{\Psi}2s{\bbR^{n}_{-}}$ but is independent of $\delta$. Next note that $\mathfrak{P}_{\vec P}(\vec M)= \vec P \vec M \vec P^{-1}$ for any matrix $\matM$, according to \eqref{eq:frakP}, whence 
\begin{align*}
\mathfrak{P}_{\vec P}(\vec M) - \matM 
& = \grad_{\xe} \Psi\matM \del{\grad_{\xe} \Psi} ^{-1} - \matM \\
& = \del{\grad_{\xe} \Psi - \matI}\matM \del{\grad_{\xe} \Psi}^{-1} 
+ \matM \del{\grad_{\xe} \Psi}^{-1}\del{\matI - \grad_{\xe} \Psi}.
\end{align*}
We recall that \eqref{eq:loc_diff} reads $\widetilde{\calE} \omega = \Psi - \calI$, whence $\norm{\calI - \grad_\xe \Psi}_{L^\infty(\Ue)} = |\widetilde{\calE}\omega|_{W^\infty(\Ue)}$, and apply the preceding expression in conjunction with \eqref{eq:omega_estimate_b} to bound the remaining terms \textrm{II} and \textrm{III} as follows:
\begin{align*}
\abs{\textrm{II}}, \ \abs{\textrm{III}}      
\le C\delta^{1-n/s}\normSZ{\we}1r{\Ue}.
\end{align*}
Collecting the estimates for \textrm{I}, \textrm{II} and \textrm{III}, and using the invertibility of $\calS_{\bbR^n_{-}}$, we obtain
\begin{align*}
\norm{\calB\of{\we, \pi}}_{\dual{X_{r'}\of{\halfspace}}} 
&\leq C \delta^{1-n/s}\norm{\of{\we, \pi}}_{X_{r}\of{\halfspace}} 
  = C \delta^{1-n/s}\norm{\calS_\halfspace^{-1}\calS_\halfspace\of{\we, \pi}}_{X_{r}\of{\halfspace}}  \\
&\leq C \delta^{1-n/s}
\norm{\calS_\halfspace^{-1}}_{\calL\of{\dual{X_{r'}\of{\halfspace}}, X_{r}\of{\halfspace}}} 
\norm{\calS_\halfspace\of{\we, \pi}}_{\dual{X_{r'}\of{\halfspace}}},
\end{align*}
where $C$ depends on $n$, $r$, $s$, $\Ue$ and $\normS{\Psi}2s{\bbR^{n}_{-}}$. 
By choosing $\delta$ small enough we satisfy the assumption of  
Lemma \ref{lem:Kato} and conclude the first part of our theorem.

To prove the second part involving further regularity in $X_{t}\of{\Qe}$ for $t > r$ we simply follow Galdi-Simader-Sohr \cite[p. 159] {GPGaldi_CGSimader_HSohr_1994}. This concludes the proof. 
\end{proof}

\subsection{$\calS_\Omega$ has finite index}
\label{s:decomposition}

In this section we prove that $\calS_\Omega$ has finite index or
equivalently, according to Lemma \ref{lem:lax_c}, that $\calS_\Omega$ has a pseudo-inverse.

\begin{lemma}[domain decomposition] 
\label{lem:dd}
Let $\Omega \subset \bbR^n$ be a bounded domain of class $W^{2-1/s}_{s}$.
There exists a finite open covering of 
$            \clos\Omega \subseteq \bigcup_{i=1}^k B\of{\vec x_i, \delta_i / 2}$, 
such that 
\begin{enumerate}[(i)]
 \item if $\vec x_i \in \Omega$ then $B\of{\vec x_i, \delta_i/2} \cap \bdy\Omega = \emptyset$.
 \item if $\vec x_i \in \bdy\Omega$ then the associated local
   $W^2_s$-diffeomorphism $\Psi_i$ is a bijection between
   $\Ue(\xe_i,\delta) = \supp\of{\calI - \Psi_i} 
   \subset \Be_{-}(\xe_i,\delta)$ and $U(\vec x_i,\delta) =
   \Psi_i(\Ue(\xe_i,\delta))$ with the disc $\De\of{\xe_i', \delta_i/2}$ being
   mapped to $\bdy\Omega \cap B\of{\vec x_i, \delta_i/2}$ (see \figref{f:RefMap}),       
i.e. it flattens the boundary of $\Omega$ near $\vec x_i$.
 \item the (perturbed) Stokes operator $\widetilde\calS_{i}$ is invertible. 
\end{enumerate}
\end{lemma}
\begin{proof} Since $\clos\Omega$ is compact, the trivial covering generated by the $\set{B\of{\vec x_i, \delta_i/2}}$, 
with $\vec x_i \in \clos\Omega$ and $\delta_i$ computed in \thmref{thm:hs_perturbed_wellposed}, has a finite sub-covering. 
Results (i)-(iii) follow immediately.
\end{proof}

We subordinate to the finite covering of \lemref{lem:dd} the following set of functions:
\begin{enumerate}[(a)]
    \item a smooth partition of unity $\set{\varphi_i}_{i=1}^k$ of $\clos\Omega$, i.e. $\varphi_i$ in $\CCinf{B\of{\vec x_i, \delta_i/2}}$ with $0 \leq \varphi_i \leq 1$,  $\sum_{i=1}^k \varphi_i\of{\vec x} = 1$ for every $\vec x \in \clos\Omega$. 
 \item
     smooth characteristic functions $\seq{\varrho_i}_{i=1}^{k}$ of $B\of{\vec x_i, \delta_i/2}$ with support on $B\of{\vec x_i, \delta_i}$, i.e. $\varrho_i \in \CCinf{B\of{\vec x_i, \delta_i}}$ with $\varrho_i = 1$ on $B\of{\vec x_i, \delta_i/2}$.
\end{enumerate}

\begin{lemma}[space decomposition of $X_r\of{\Omega}$] 
\label{lem:space_decomposition}
Let $s' \leq r \leq s$. The following identities hold
    \[
         \calI_{X_r\of{\Omega}} = \sum_{i=1}^k \calR_{\varphi_i}\calE_{\varrho_i}
         ,\quad 
         \calI_{\dual{X_r\of{\Omega}}} = \sum_{i=1}^k \dual\calE_{\varrho_i}\dual{\calR}_{\varphi_i}.
    \]
As a result we may decompose $X_r\of\Omega$ and $\dual{X_{r}\of\Omega}$ as follows:
    \[
        X_r\of{\Omega} =\sum_{i=1}^k \calR_{\varphi_i} X_r\of{\Qe_i}
        ,\quad 
        \dual{X_r\of{\Omega}} = \sum_{i=1}^k \dual{\calE}_{\varphi_i} \dual{X_r\of{\Qe_i}} , 
    \]
where $\Qe_i = \bbR^n$ or $\bbR^n_{-}$. 
\end{lemma}
\begin{proof}
\noindent We begin by proving the relation for $\calI_{X_r\of{\Omega}}$ and simultaneously show the decomposition of $X_r(\Omega)$. 
Let $\of{\vec u, p}$ in $X_r\of{\Omega}$ be fixed but arbitrary. Using that $\varphi_i \varrho_i = \varphi_i$ and $\sum_{i=1}^k \varphi_i = 1$,
Definition~\ref{def:loc_op} implies
\[
\of{\vec u, p} = \sum_{i=1}^{k} \varphi_i\varrho_i\of{\vec u, p} 
= \sum_{i=1}^{k} \varphi_i \calP_i \calP_i^{-1}\of{\varrho_i\vec u, \varrho_i p} 
= \sum_{i=1}^{k} \calR_{\varphi_i}\calE_{\varrho_i}\of{\vec u, p} .              
\]    
This proves the identity relation for $\calI_{X_r\of{\Omega}}$. 

Since the operators $\calE_{\varrho_i} : X_r(\Omega) \rightarrow X_r(\Qe_i)$
are continuous, due to \lemref{lem:cont_RE}, it follows that 
the ``vector'' $(\calE_{\varrho_i}\of{\vec u, p})_{i=1}^k \in \Pi_{i=1}^k X_r\of{\Qe_i}$. 
Conversely, if $\{{\of{\ue_i, \pe_i}} \}_{i=1}^k \in \Pi_{i=1}^k X_r(\Qe_i)$, 
then the continuity of $\calR_{\varphi_i} : X_r(\Qe_i) \rightarrow X_r(\Omega)$, implies $\calR_{\varphi_i}(\ue_i, \pe_i) \in X_r(\Omega)$ and 
$\sum_{i=1}^k \calR_{\varphi_i}(\ue_i, \pe_i) \in X_r(\Omega)$ because
$k$ is finite. This implies the decomposition of $X_r(\Omega)$. 

The remaining decompositions of $\calI_{\dual{X_r\of{\Omega}}}$ and
$\dual{X_r\of{\Omega}}$ can be proven similarly.
\end{proof}

\begin{theorem}
[pseudoinverse of $\calS_\Omega$] 
\label{thm:pinvS}
Let $s' \leq r \leq s$, $\calS_\Omega : X_r(\Omega) \rightarrow X_{r'}(\Omega)^*$ be Stokes operator defined in \eqref{eq:stokes_abstract}, and $\widetilde\calS_i : X_r(\Qe_i) \rightarrow X_{r'}(\Qe_i)^*$ be the perturbed Stokes operator defined in Propositions~\ref{prop:sslip_interior_localization} and \ref{prop:sslip_boundary_localization}. The operator $\calS_\Omega^\dagger : \dual{X_{r'}\of{\Omega}} \to X_{r}\of\Omega$ 
    \[
        \calS_\Omega^\dagger :=  \sum_{i=1}^{k} 
            \calR_{\varrho_i}\widetilde\calS_{i}^{-1}\dual{\calR}_{\varphi_i} 
    \]
is a pseudoinverse of $\calS_\Omega$.
\end{theorem}
\begin{proof} 
To simplify the notation take $\calC_i = 0$ whenever $\vec x_i \in \Omega$. In view of Propositions~\ref{prop:sslip_interior_localization} and \ref{prop:sslip_boundary_localization}, we can write 
    \begin{align*}
        \calS_\Omega^\dagger \calS_\Omega 
            &= \sum_{i=1}^k  
                \calR_{\varrho_i}\widetilde\calS_{i}^{-1}
                        \of{\widetilde S_{i}\calE_{\varphi_i} + \calC_i\calE_{\varphi_i} 
                      + \dual{\calP}_{i}\calK_{\varphi_i}} \\
            &= \sum_{i=1}^k \calR_{\varrho_i} \calE_{\varphi_i} 
             + \sum_{i=1}^{k}\calR_{\varrho_i}\widetilde\calS_{i}^{-1}\of{\calC_i\calE_{\varphi_i} + \dual{\calP}_{i}\calK_{\varphi_i}} \\
            &= \calI_{X_r\of{\Omega}}  
               + \sum_{i=1}^k  \calR_{\varrho_i}\widetilde\calS_{i}^{-1}\of{\calC_i\calE_{\varphi_i} + \dual{\calP}_{i}\calK_{\varphi_i}}. 
    \end{align*}
Since $\calC_i$ and $\calK_{\varphi_i}$ are compact, according to Lemmas~\ref{lem:k_compact} and \ref{lem:C_compact}, and so is
$\calS_\Omega^\dagger \calS_\Omega - \calI_{X_r\of{\Omega}}$ \cite[Chapter 21: Theorem 1]{PDLax_2002}. The proof that $\calS_\Omega \calS_\Omega^\dagger - \calI_{\dual{X_{r}\of{\Omega}}}$ is compact follows along the same lines.
\end{proof}

\subsection{$\calS_\Omega$ and $\dual\calS_\Omega$ are injective}
\label{s:stokes_slip_injective}

In view of our strategy to use Corollary~\ref{cor:h_r_s} to infer the invertibility of the Stokes operator $\calS_\Omega$, it is essential to prove the injectivity of $\calS_\Omega$ and $\dual\calS_\Omega$. This is precisely the aim of this section. We proceed by making an immediate observation. 

\begin{proposition}
\label{prop:s_omega_is_injective} 
Let $2 \leq r \leq s$. The Stokes operator $\calS_\Omega: X_r\of{\Omega} \to \dual {X_{r'}\of\Omega}$ is injective.
\end{proposition}
\begin{proof} Owing to \thmref{thm:hilbert_case}, if $(\vec u,p) \in X_2\of\Omega$ solves $\calS_\Omega(\vec u,p) = 0$, 
then $(\vec u,p) = (\vec 0,0)$. To prove the assertion for $r > 2$ we use that $\Omega$ is bounded so that 
$\cembed{X_r\of\Omega}{X_2\of\Omega}$, and as a consequence any solution $\del{\vec u, p} \in X_r\of\Omega$ of 
$\calS_\Omega(\vec u, p) = 0$  necessarily belongs to $X_2\of\Omega$. We conclude that $\calS_\Omega$ is injective 
for $2 \leq r \leq s$. 
\end{proof}

To prove that $\calS_\Omega$ is injective for $s' \leq r < 2$ we will develop an induction argument; for a somewhat related result see \cite[pg. 159] {GPGaldi_CGSimader_HSohr_1994}. It seems that the induction argument is rooted in the work introduced by Moser in the context of elliptic PDE \cite{JMoser_1960a}\cite[Section 8.5]{DGilbarg_NTrudinger_2001a}. Leveraging the boundedness of $\Omega$ it is sufficient to prove this argument for $r=s'$. We will show in a finite number of steps that a homogeneous solution in $X_{s'}\of\Omega$ is in fact in $X_{t}\of\Omega$ for some $t \geq 2$.

\begin{definition}[smoothing sequence] 
\label{def:smooth_seq}
Let $s > n$ and $t_{-1} = r = s'$. We introduce a smoothing sequence  
    \begin{align*}
        \frac 1 {t_0} &:= 1 - \frac 2 {s+n} \\
        \frac 1 {t_m} &:= \frac 1 {t_{m-1}} + \frac{1}{s} - \frac 1 n \quad  \mbox{for } m=1,\ldots, M
    \end{align*}
where $M = \ceil{\del{\frac{1}{n}-\frac{1}{s}}^{-1}\del{\frac{1}{2}-\frac{2}{s+n}}}$ guarantees that $t_M \geq 2$. 
\end{definition}

\begin{remark}[properties of $t_m$]\label{rem:tmprime} \rm
We observe that $t_m$ is monotone increasing because $s>n$. 
Moreover, $t_{m-1} < 2 \le n$ implies $t_m' < n$ for $m \geq 1$. In fact,
the final conclusion is due to the definitions of $t_m$ and $t_m'$, namely
\[
\frac{1}{t_m'} = 1 - \frac{1}{t_m} =
\frac{1}{s'} - \frac{1}{t_{m-1}} + \frac{1}{n},
\]
and that $t_{m-1}$ is monotone increasing with $t_{m-1} > s' = t_{-1}$.
\end{remark}

\begin{lemma}[Sobolev embedding] 
\label{lem:injectivity_embedding_s}
Let $\calD \subset \bbR^n$ be a bounded Lipschitz domain. The
following holds for $m \geq 0$ and $t_{m-1} < 2 \le n$ 
   \begin{align*}
        \cembed{\sob{1}{t_{m-1}}\calD}{L^{t_m}\of{\calD}} 
        ,\quad 
        \cembed{\sob{1}{t_m'}\calD }{L^{t_{m-1}'}\of{\calD}} . 
   \end{align*}
Equivalently, for every $u$ in $\sob{1}{t_{m-1}}\calD$ and $v$ in $\sob1{t_m'}\calD$,
    \begin{equation}\begin{aligned}\label{eq:embd}
        \normL{u}{t_m}{\calD} 
            &\leq C_{m,n,s,\calD}\normS{u}1{t_{m-1}}\calD \\
        \normL{v}{t_{m-1}'}{\calD}
            &\leq C_{m,n,s,\calD}\normS{v}1{t_{m}'}\calD . 
    \end{aligned}\end{equation}
\end{lemma}
\begin{proof} We split the proof into two cases. We recall the Sobolev number $t^*$ associated with $t < n$: 
    \[
        \frac{1}{t^*} = \frac{1}{t} - \frac{1}{n} . 
    \] 

     \noindent \boxed{1} Case $m = 0$:
            Since $t_{-1} = s' < 2 \leq n$, ${\sob{1}{t_{-1}}\calD} \hookrightarrow {L^{t_0}\of{\calD}}$ if 
            $t_{0} \leq t_{-1}^*$. This is true because 
            \[
                \frac{1}{t_0} - \frac{1}{t^*_{-1}} 
                    = 1 - \frac{2}{s+n} - \del{\frac{1}{t_{-1}} - \frac 1 n}
                    = \frac{1}{s} + \frac 1 n - \frac{2}{s+n}
                    = \frac{n^2 + s^2}{sn\of{s+n}} > 0. 
            \]
            For the second embedding we note $t_{-1}' = s$, and $t_{0}' = (s+n)/2 > n$, thus we have 
            $\cembed{\sob{1}{t_{0}'}D}{{L^{t_{-1}'}\of{D}}}.$

        \noindent \boxed{2} Case $m \geq 1$:
            Since $t_{m-1} < 2 \leq n$ and $\cembed{\sob1{t_{m-1}}D}{L^{t^*_{m-1}}(D)}$, to get the first
            embedding, it suffices to verify  $t_m \leq t_{m-1}^*$:
            \[
               \frac{1}{t_m} - \frac{1}{t_{m-1}^*} 
                   = \frac{1}{t_{m-1}} + \frac{1}{s} - \frac{1}{n} - \del{\frac{1}{t_{m-1}} - \frac{1}{n}} 
                   = \frac{1}{s} \geq 0.
            \]
        Using Definition~\ref{def:smooth_seq}, we deduce $t_{m-1}' \leq (t_m')^*$ whence $\cembed{W^{1}_{t_m'}(D)}{L^{t_{m-1}'}(D)}$. This completes the proof.
\end{proof}

\begin{lemma}[interior regularity of homogeneous solution] 
\label{lem:interior_regularity}
Let $\vec x \in \Omega$. 
If $\of{\vec u, p} \in X_{t_{m-1}}\of{\Omega}$ satisfies $\calS_\Omega(\vec u,p) = 0$, then 
$\calE_{\zeta}\of{\vec u, p} \in X_{t_{m}}\of{\bbR^n}$.
\end{lemma}
\begin{proof} 
Since $\calS_\Omega\of{\vec u, p} = 0$ we have from
\eqref{eq:stokes_interior_decomposition_short} that
\[
\widetilde\calS \calE_{\zeta}(\vec u,p) 
= - \calP^* \calK_{\zeta}(\vec u,p).
\]
As $\calP$ is an isomorphism (cf. \thmref{thm:bhs_isomorphic_domains})
the strategy is show that $\calK_{\zeta}(\vec u,p) \in X_{t_m'}(\bbR^n)^*$ 
provided $(\vec u,p) \in X_{t_{m-1}}(\bbR^n)$ and use  \thmref{thm:hs_perturbed_wellposed} 
to conclude that $\calE_\zeta (\vec u,p) \in X_{t^m}(\bbR^n)$. 
For simplicity we use the notation $U = U(\vec x,\delta)$. 

Let $\of{\vec v, q} \in X_{t_m'}\of{\bbR^n}$. In view of the
definition \eqref{eq:stokes_localized_fi},
we split  $\calK_{\zeta}\of{\vec u, p}$ into four terms
\begin{gather*}
  \textrm{I} = -\pair{p, \grad_{\vec x}\zeta\cdot \vec v}_{U} , \quad 
  \textrm{II} = -\pair{\grad_{\vec x} \zeta \cdot \vec u, q}_{U} , \\
  \textrm{III} = -\pair{\vec{\mathfrak{e}}_\zeta\of{\vec u}, \gradS{\vec v}}_{U}, 
                  \quad 
  \textrm{IV} = \pair{\gradS{\vec u}, \vec{\mathfrak{e}}_\zeta\of{\vec v}}_{U} ,  
\end{gather*}
and estimate them separately. Invoking H\"older's inequality and \eqref{eq:embd}, we obtain 
    \begin{align*}
        \abs{\textrm{I}}  
            &\lesssim \normL{p}{t_{m-1}}{U}
                    \normL{\vec v}{t_{m-1}'}{U}
             \lesssim \normL{p}{t_{m-1}}{U}
                    \normS{\vec v}1{t_{m}'}{U}, \\       
        \abs{\textrm{II}}            
            &\lesssim \normL{\vec u}{t_{m}}{U} 
                    \normL{q}{t_{m}'}{U} 
             \lesssim \normS{\vec u}1{t_{m-1}}{U}  \normL{q}{t_{m}'}{U} .
       \end{align*}
Using H\"older's inequality again, this time in conjunction with
\eqref{eq:stokes_interio_theta_estimate} and \eqref{eq:embd}, we
  see that
       \begin{align*}
        \abs{\textrm{III}}                   
            &\lesssim \normL{\vec u}{t_{m}}{U} 
                    \normL{\gradS{\vec v}}{t_{m}'}{U} 
             \lesssim \normS{\vec u}1{t_{m-1}}{U} 
                    \normS{\vec v}1{t_{m}'}{U} \\ 
        \abs{\textrm{IV}}                             
            &\lesssim \normL{\gradS {\vec u}}{t_{m-1}}{U} 
                    \normL{\vec v}{t_{m-1}'}{U}
             \lesssim \normL{\vec u}{t_{m-1}}{U} 
                    \normS{\vec v}1{t_{m}'}{U} .        
    \end{align*}
    
Since $\normS{\vec v}1{t_{m}'}{U} \lesssim \normS{\vec v}1{t_{m}'}{\Omega} \lesssim \normSZ{\vec v}1{t_{m}'}{\bbR^n}$ the latter being the norm of $V_{t_m'}(\bbR^n)$ according to \eqref{eq:gs_space} and \eqref{eq:banach_space_X_hom}, we deduce 
\[
\abs{\calK_{\zeta}\of{\vec u, p}\of{\vec v, p}}
\lesssim \Big( \normS{\vec u}1{t_{m-1}}{\Omega} + \normL{p}{t_{m-1}}{\Omega} \Big)
\norm{\del{\vec v, q}}_{X_{t_{m}'}(\bbR^n)} . 
\]
Finally, using \lemref{lem:equiv_norm} (Korn's inequality) namely  
$\normL{\vec u}{t_m}{U} \lesssim \normL{\gradS {\vec u}}{t_{m-1}}{\Omega} \lesssim \normSZ{\vec u}1{t_{m-1}}\Omega$, we obtain 
$$
\norm{\calK_{\zeta}\del{\vec u,p}}_{X_{t_m'}(\bbR^n)^*} \le \norm{\del{\vec u, p}}_{X_{t_{m-1}}(\Omega)} . 
$$ 
This completes the proof. 
\end{proof}

\begin{lemma}[boundary regularity of homogeneous solution]
\label{lem:boundary_regularity}
Let $\vec x \in \bdy\Omega$. If $\of{\vec u, p} \in X_{t_{m-1}}\of{\Omega}$ satisfies 
$\calS_\Omega(\vec u, p) = 0$, then $\calE_{\zeta}\of{\vec u, p} \in X_{t_{m}}\of{\bbR^{n}_{-}}$.
\end{lemma}
\begin{proof} 
Since $\calS_\Omega\of{\vec u, p} = 0$ we have from \eqref{eq:stokes_boundary_decomposition_short}
    \[
        \widetilde\calS\calE_{\zeta}\of{\vec u, p} 
            =   - \calC \calE_{\zeta}\of{\vec u, p}
                - \dual{\calP}\calK_{\zeta}\of{\vec u, p}.
    \]
The strategy is the same as in \lemref{lem:interior_regularity}: we
show that the right-hand-side is in 
$\dual{X_{t_m'}\of{\bbR^{n}_{-}}}$ provided $(\vec u, p) \in X_{t_{m-1}}(\bbR^n_{-})$, and use 
\thmref{thm:hs_perturbed_wellposed} to conclude that $\calE_\zeta(\vec u, p) \in X_{t_m}(\bbR^n_{-})$. 
In particular, the regularity for $K_{\zeta}\of{\vec u, p}$ follows from the exact same argument 
used in \lemref{lem:interior_regularity}. 
We only need to prove the additional regularity for $\calC\calE_{\zeta}\of{\vec u, p}$. 

Let $\of{\ve, \qe} \in X_{t_m'}\of{\bbR^{n}_{-}}$ be fixed but
arbitrary and set $\of{\ue, \pe} = \calE_{\zeta}\of{\vec u, p}$. In
view of definition \eqref{eq:s_decomp_ci},
we split $\calC \of{\ue, \pe} \of{\ve, \qe}$ into three terms
\begin{gather*}
  \textrm{I} = \big\langle \vec\vartheta_{\matP}\of{\ue}, J_{\xe}^{-1}
         \vec\vartheta_{\matP}\of{\ve} \big\rangle_{\Ue}, \quad          
  \textrm{II} = -\big\langle \vec\vartheta_{\matP}\of{\ue},  J_{\xe}^{-1}
    \vec{\varepsilon}_{\matP}\of{\ve} \big\rangle_{\Ue}  , \\
  \textrm{III} = -\big\langle \vec{\varepsilon}_{\matP}\of{\ue}, J_{\xe}^{-1}
    \vec\vartheta_{\matP}\of{\ve}\big\rangle_{\Ue}    , 
\end{gather*}
where $U = U(\vec x,\delta) = \Psi(\Ue)$ is the physical domain
and $\Ue$ is the bubble domain (see Figure \ref{f:RefMap}).
We also recall the Sobolev number $t^*$ 
associated with $t < n$: 
\[
\frac{1}{t^*} = \frac{1}{t} - \frac{1}{n} . 
\] 
We split the proof into two parts depending on whether $m \geq 1$ or $m=0$. 

\noindent \boxed{1} Case $m \geq 1$: Since $t_{m-1} < 2 \le n$ we have 
$\ue \in L^{t^*_{m-1}}(\Ue)$, and applying \eqref{eq:piola_sym_grad_bound} 
with 
\[
\frac{1}{t_{m-1}^\bullet} 
= \frac{1}{s} + \frac{1}{t_{m-1}^*} 
= \frac{1}{t_{m-1}} + \frac{1}{s} - \frac{1}{n} = \frac{1}{t_m}
\]
we deduce $\vec\vartheta_{\vec P}(\ue) \in L^{t_m}(\Ue)$ with 
\begin{equation}\label{lem:boundary_regularity_a}
\normL{\vec\vartheta_{\vec P}(\ue)}{t_m}{\Ue}
\lesssim \normS{\ue}1{t_{m-1}}\Ue , 
\end{equation}
as a consequence of \eqref{eq:embd}. 
On the other hand, combining Remark~\ref{rem:tmprime} with
\eqref{eq:GN_inequality} yields $\ve \in L^{(t_m')^*}(\Ue)$, whence invoking 
\eqref{eq:piola_sym_grad_bound} with 
\[
\frac{1}{(t'_m)^\bullet} 
= \frac{1}{(t'_m)^*} + \frac{1}{s} 
= \frac{1}{t'_{m-1}} , 
\]
noticing that $t'_m < t'_{m-1}$
and using Gagliardo-Nirenberg inequality, we obtain 
\[
\normL{\vec\vartheta_{\vec P}(\ve)}{t_m'}{\Ue}
\lesssim \normL{\ve}{(t'_m)^*}{\Ue}
\lesssim \normSZ{\ve}1{t_{m}'}\Ue .
\]
This implies
\[
\abs{\textrm{I}}
\lesssim \normL{\vec\vartheta_{\vec P}\of{\ue}}{t_m}\Ue 
    \normL{\vec\vartheta_{\vec P}\of{\ve}}{t_m'}\Ue         
\lesssim \normS{\ue}1{t_{m-1}}\Ue
    \normSZ{\ve}1{t_{m}'}\Ue .     
\]
Using \lemref{lem:equiv_norm} (Korn's inequality), namely 
$\normS{\ue}1{t_{m-1}}\Ue \lesssim \normL{\vec\varepsilon \of{\vec u}}{t_{m-1}}\Omega  \lesssim \linebreak \normSZ{\vec u}1{t_{m-1}}\Omega$, 
we arrive at
\[
\abs{\textrm{I}} 
\lesssim \normSZ{\vec u}1{t_{m-1}}\Omega
    \normSZ{\ve}1{t_{m}'}\Ue.  
\]

To estimate \textrm{II} we resort to \eqref{eq:piola_sym_grad_bound}, 
\eqref{lem:boundary_regularity_a}, and \lemref{lem:equiv_norm}, to deduce 
\[
\abs{\textrm{II}} 
\lesssim  \normL{\vec\vartheta_{\vec P}\of{\ue}}{t_m}\Ue
    \normL{\vec\varepsilon_{\vec P}\of{\ve}}{t_m'}\Ue    
\lesssim  \normS{\ue}1{t_{m-1}}\Ue
    \normSZ{\ve}1{t_{m}'}\Ue    
\lesssim 
    \normSZ{\vec u}1{t_{m-1}}\Omega \normSZ{\ve}1{t_{m}'}\Ue  . 
\]

We next deal with $\textrm{III}$. Since 
\[
\frac{1}{t'_{m-1}} - \frac{1}{s} = \frac{1}{t'_m} - \frac{1}{n} = \frac{1}{(t'_m)^*}  > 0,
\]
applying \eqref{eq:piola_sym_grad_bound} yields $\normL{\vec\vartheta_{\vec P}\of{\ve}}{t_{m-1}'}\Ue \lesssim \normL{\ve}{(t'_m)^*}\Ue
\lesssim \normSZ{\ve}1{t'_m}\Ue$, where the last inequality is due to Gagliardo-Nirenberg inequality. \lemref{lem:equiv_norm} (Korn's inequality) further leads to
\[
\abs{\textrm{III}} 
\lesssim \normL{{\vec\varepsilon}_{\vec P} \of{\ue}}{t_{m-1}}\Ue
    \norm{\vec\vartheta_{\matP}\of{\ve}}_{L^{t_{m-1}'}(\Ue)} 
\lesssim 
    \normSZ{\ue}1{t_{m-1}}\Ue \normSZ{\ve}1{t_m'}\Ue     
\]
%
%
Combining the estimates for \textrm{I}, \textrm{II}, and \textrm{III} gives
\[
\abs{\calC\of{\ue, \pe}\of{\ve, \qe}}
\lesssim \norm{\of{\vec u, p}}_{X_{t_{m-1}}(\Omega)} \norm{\of{\ve,\qe}}_{X_{t_{m}'}(\bbR^{n}_{-})} , 
\]
thereby leading to $\calC\of{\ue, \pe} \in X_{t_{m}'}(\bbR^{n}_{-})^*$ as desired.

\noindent \boxed{2} Case $m=0$: Since $t_{-1} = s' < n$ we have $\ue \in L^{t^*_{-1}}(\Ue)$, whence $\vec\vartheta_{\vec P}(\ue) \in L^{t^\bullet_{-1}}(\Ue)$ with 
\[
\frac{1}{t^\bullet_{-1}} 
= \frac{1}{s} + \frac{1}{t^*_{-1}} 
= \frac{1}{s} + \frac{1}{s'} - \frac{1}{n}
= 1 - \frac{1}{n} 
= \frac{1}{n'} 
\]
and 
\[
\normL{\vec\vartheta_{\vec P}(\ue)}{n'}\Ue
\lesssim \normS{\ue}1{t_{-1}}\Ue  ,
\]
because of \eqref{eq:piola_sym_grad_bound}. On the other hand, the fact that 
\[
\frac{1}{t_0'} = 1 - \frac{1}{t_0} = \frac{2}{s+n} < \frac{1}{n} 
\]
implies 
\[
\normL{\vec\vartheta_{\vec P}(\ve)}{n}\Ue 
\lesssim \normL{\ve}{\infty}\Ue 
\lesssim \normS{\ve}1{t_0'}\Ue .
\]
Consequently
\[
\abs{\textrm{I}} 
\lesssim \normL{\vec\vartheta_{\vec P}(\ue)}{n'}\Ue
    \normL{\vec\vartheta_{\vec P}(\ue)}{n}\Ue
\lesssim \normS{\ue}1{t_{-1}}\Ue 
    \normS{\ve}1{t'_0}\Ue . 
\]
Since $\ve \in \sob1{t_0'}\Ue$, we see that
\[
\abs{\textrm{II}} 
\lesssim \normL{\vec\vartheta_{\vec P}(\ue)}{n'}\Ue
    \normL{\vec\varepsilon_{\vec P}(\ve)}{n}\Ue 
\lesssim \normS{\ue}1{t_{-1}}\Ue \normSZ{\ve}1{n}\Ue     
\lesssim \normS{\ue}1{t_{-1}}\Ue \normSZ{\ve}1{t_0'}\Ue .    
\]
As $\ue \in \sob1{t_{-1}}\Ue$ and $t_{-1} = s'$, with the help of \eqref{eq:piola_sym_grad_bound} for $t^\circ = \infty$ we infer that 
\[
\abs{\textrm{III}} 
\lesssim \normL{\vec\varepsilon_{\vec P}(\ue)}{s'}\Ue
    \normL{\vec\vartheta_{\vec P}(\ve)}{s}\Ue
\lesssim \normSZ{\ue}1{t_{-1}}\Ue     
    \normL{\ve}\infty\Ue 
\lesssim \normSZ{\ue}1{t_{-1}}\Ue \normS{\ve}1{t_0'}\Ue .  
\]
Applying \lemref{lem:equiv_norm} (Korn's inequality), we deduce 
$\normS{\ue}1{t_{-1}}\Ue \lesssim \normSZ{\ue}1{t_{-1}}\Ue$, and further using 
Remark~\ref{rem:norm_equiv} we obtain $\normS{\ve}1{t_0'}\Ue \lesssim \normSZ{\ve}1{t_0'}\Ue$. 
This implies that  
\[
\abs{\calC\of{\ue, \pe}\of{\ve, \qe}}
\lesssim \norm{\of{\ue,\pe}}_{X_{t_{-1}}(\Ue)} \norm{\of{\ve,\qe}}_{X_{t_0'}(\bbR^{n}_{-})} , 
\]
and as a consequence that $\calC\of{\ue, \pe} \in X_{t_0'}(\bbR^{n}_{-})^*$. This concludes the proof. 
\end{proof}

\begin{corollary}[global regularity]
\label{cor:glo_reg}
If $\of{\vec u, p}$ in $X_{t_{m-1}}\of{\Omega}$ satisfies $\calS(\vec u,p) = 0$, then $\of{\vec u, p} \in X_{t_{m}}\of{\Omega}$.
\end{corollary}
\begin{proof} 
For a given $(\vec u,p) \in X_{t_{m-1}}(\Omega)$, \lemref{lem:space_decomposition} (space decomposition of $X_r(\Omega)$) with $r = t_{m-1}$ implies 
\[ 
\of{\vec u, p} =\sum_{i=1}^{k} \calR_{\varphi_i}\calE_{\varrho_i} \of{\vec u, p} .
\]
For $i = 1,\dots, k$, Lemmas~\ref{lem:interior_regularity} and \ref{lem:boundary_regularity}, yield
$\calE_{\varrho_i}\of{\vec u, p} \in X_{t_m}\of{\Qe_i}$ with $\Qe_i = \bbR^n$ or $\bbR^n_{-}$.  
Finally, using \lemref{lem:cont_RE} (continuity of $\calR_{\varphi_i}$)
leads to the asserted result.
\end{proof}

\begin{proposition}
[injectivity of $\calS_{\Omega}$]
\label{s:sslip_somegar_injective}
Let $s' \leq r < 2$. The Stokes operator $\calS_{\Omega} : X_r\of\Omega \to \dual{X_{r'}\of\Omega}$ is injective.
\end{proposition}
\begin{proof} 
As $\Omega$ is bounded, it suffices to prove the assertion for $r = s'$. 
Let $(\vec u, p) \in X_{s'}\of\Omega$ solve $\calS_\Omega(\vec u,p) = 0$. 
Applying Corollary~\ref{cor:glo_reg} $M$ times, where $M$ is given in Definition~\ref{def:smooth_seq}, we obtain 
    \[
        \of{\vec u, p} \in X_{t_{-1}=s'}\of\Omega \cap \ldots \cap X_{t_M}\of\Omega,
    \]
with $t_M \geq 2$. \propref{prop:s_omega_is_injective} further implies $(\vec u, p) = (\vec 0,0)$, which concludes the
proof.
\end{proof}

\begin{proposition}[injectivity of $\dual\calS_{\Omega}$]
\label{prop:dSinj}
Let $s' \leq r \leq s$. The dual Stokes operator $\dual\calS_{\Omega}$ is injective.
\end{proposition}
\begin{proof} 
We use the subscript $r$ on the operator $\calS_{\Omega}$ to indicate that $\calS_\Omega$ is defined on $X_r(\Omega)$, 
i.e. $\calS_{\Omega,r} : X_r(\Omega) \rightarrow X_{r'}(\Omega)^*$. Since $X_{r'}(\Omega)$ is reflexive, we further 
deduce $\calS_{\Omega,r}^* : X_{r'}(\Omega) \rightarrow X_{r}(\Omega)^*$. 
Let $(\vec v, q) \in X_{r'}(\Omega)$ satisfy 
\[
\calS_{\Omega,r}^*(\vec v,q)(\vec u,p) = 0, \quad \forall (\vec u, p) \in X_{r}(\Omega) .
\]
Using the definition of adjoint operator  
\[
\calS_{\Omega,r}(\vec u,p)(\vec v,q) = 0, \quad \forall (\vec u, p) \in X_{r}(\Omega) .
\] 
Owing to the definition of $\calS_{\Omega,r}$ we have
\[
\calS_{\Omega,r'}(\vec v,-q)(\vec u, -p) = 0, \quad \forall (\vec u, p) \in X_{r}(\Omega) .
\]
As $\calS_{\Omega,r'}$ is injective, thus $(\vec v, q) = (\vec 0,0)$, which completes the proof. 
\end{proof}


\begin{corollary}[invertibility of $\calS_\Omega$]
The Stokes operator $\calS_\Omega : X_r(\Omega) \rightarrow
X_{r'}(\Omega)^*$ is invertible for $s' \leq r \leq s$.
\end{corollary}
\begin{proof}
The index of $\calS_\Omega$ is finite because $\calS_\Omega$ has a pseudoinverse (\lemref{lem:lax_c} and \thmref{thm:pinvS}). In addition, $\calS_\Omega$ is injective (Proposition~\ref{prop:s_omega_is_injective} and \ref{s:sslip_somegar_injective}) and $\calS_\Omega^*$ is also injective (Proposition~\ref{prop:dSinj}). Apply Corollary~\ref{cor:h_r_s} to conclude the assertion. 
\end{proof}


\section{The non-homogeneous case \boldmath$u \cdot \nu \neq 0$}
\label{s:lifting}
We present a framework to treat the nonhomogeneous essential boundary condition \eqref{eq:stokes_bc}. It relies on the standard practice of lifting the data inside the domain. By the principle of superposition it suffices  to study the case when $\phi$ is the only non-trivial data. Note that the compatibility condition \eqref{eq:stokes_compatibility} becomes $\int_{\bdy\Omega} \phi = 0$.

Since the domain is of class $W^{2-1/s}_s$, for $s>n$, the unit
normal satisfies $\vec\nu \in \sob{1-1/s}{s}{\bdy\Omega}$.
We extend each component of $\vec\nu$ and thus $\vec\nu$
to a function in $W^1_s(\Omega)$ (still denoted $\vec \nu$).
Given a scalar function $\phi \in \sob{1-1/r}{r}{\bdy\Omega}$, we
still denote $\phi\in W^1_r(\Omega)$ its extension to $\Omega$. These
extensions are possible because $\partial\Omega$ is Lipschitz. We
define $\vec\varphi := \vec\nu \phi$ and note that, for
$r\le s$ and $s>n$, a simple calculation yields
\begin{equation}\label{eq:sslip_phi_lift_bound}
        \normS{\vec\varphi}1r\Omega 
            \leq C_{\Omega,n,r,s} 
                \normS{\phi}{1-1/r}r{\bdy\Omega}
                \normS{\vnu}{1-1/s}{s}{\bdy\Omega} .       
\end{equation}
In fact, $\vec\nu \nabla\phi \in L^r(\Omega)$ is obviously valid.
Therefore, the only problematic term is $\phi\nabla\vec\nu$ when $r\le n$,
for otherwise $\phi$ is bounded and $\phi\nabla\vec\nu\in L^s(\Omega)$.
In such a case, $\phi\in L^t(\Omega)$ with $\frac{1}{t} > \frac{1}{r} - \frac{1}{n}$,
whence $\frac{1}{t^*}=\frac{1}{r}-\frac{1}{t}<\frac{1}{n}$. Since
$s>n$ we can choose $t$ so that $n<t^*\le s$ and
H\"older's inequality gives $\phi\nabla\vec\nu \in W^1_r(\Omega)$.
Similar estimates are reported in \cite[Lemma 13.3]{LTartar_2007} and
\cite[Corollary 1.1]{VGirault_PARaviart_1986a}\cite[Theorem
  7.39]{RAAdams_JJFFournier_2003}.

\begin{corollary}\label{cor:lifting}
The pair $\del{\vec u, p}\in X_r(\Omega)$ is a
solution to \eqref{eq:stokes_full} with only $\phi$ non-trivial if and
only if $\del{\vec w, p} = \del{\vec u - \vec\varphi, p}\in X_r(\Omega)$
is a weak solution to \eqref{eq:stokes_full} with
	\[
		\vec{f} =\eta \divg{\vec\varepsilon\of{\vec\varphi}},\quad
		g = -\divg{\vec\varphi}, \quad
		\vec\psi = -\eta \matT\transpose \vec\varepsilon\of{\vec\varphi}\vnu, \quad 
		\vec w\cdot\vnu =0.
	\]
In particular, $\del{\vec w, p}$ satisfies  
	\begin{equation*}
		\calS_\Omega\of{\vec w, p}\of{\vec v, q} 
             = \eta\pair{\vec\varepsilon\of{\vec\varphi}, \gradS
                  {v}}_\Omega 
                + \pair{p, \divg \vec v} 
		+ \pair{\vec \psi, \gamma_0 \vec v}_{\bdy\Omega}
		- \pair{\divg\vec\varphi, q}_\Omega
	\end{equation*}
for all $\of{\vec v, q} \in X_{r'}\of\Omega$,	
and its norm is controlled by the data $\phi$, namely
	\[
		\norm{\of{\vec w,p}}_{X_r\of\Omega} \leq C_{\Omega,n,r,\eta}\normS{\phi}{1-1/r}r{\bdy\Omega}.
	\]
\end{corollary}
\begin{proof} The expressions for $\vec f$, $g$, $\vec\psi$, and $\vec w\cdot\vnu$ are straightforward to obtain, so we skip their derivation. The variational form follows after integrating by parts and recalling that the test functions $\vec v$ are tangential on the boundary. 
Finally, the continuity estimate is a direct application of the results when $\vec w \cdot \vnu = 0$ as well as \eqref{eq:sslip_phi_lift_bound}.
\end{proof}

\begin{remark}[alternative lift] \rm
Given data $\phi$ in $\sob{1-1/r}{r}{\bdy\Omega}$, consider solving the Neumann problem 
    \begin{equation}\label{eq:Np}
    \begin{aligned}
        -\lapl \varphi &= 0 \quad  \mbox{in } \Omega \\
        \partial_\vnu \varphi &= \phi  \quad \mbox{on } \bdy\Omega
    \end{aligned}    
    \end{equation}
in $\sob{2}{r}{\Omega}$. Then, the pair $\of{\vec u, p}$ is a solution
of \eqref{eq:stokes}, \eqref{eq:stokes_bc}, with only $\phi$ non-trivial, if and only if $\of{\vec w, p} = \del{\vec u - \grad\varphi, p}$ is a solution of \eqref{eq:stokes}, \eqref{eq:stokes_bc}, with 
    \[
		\vec{f} =\eta \divg{\vec\varepsilon\of{\grad\varphi}}\mbox{, }\quad 
		g = -\divg{\grad\varphi} \mbox{, }\quad
		\vec\psi = -\eta \matT\transpose\vec\varepsilon\of{\grad\varphi}\vnu,\quad
		\vec w\cdot\vnu =0.    
    \]

We point out that existence of a strong solution to the
inhomogeneous Neumann problem \eqref{eq:Np} for Lipschitz or even
$C^1$ domains may fail in general; see \cite{DJerison_CEKenig_1989a}.
On the other hand, it is well-known that for $C^{1,1}$ domains a strong solution always exists \cite{MR3396210}. In this respect, our fractional Sobolev domain regularity appears to be (nearly) optimal.
\end{remark}

\section{The Navier boundary condition}
\label{s:friction_bc}
The goal of this section is to consider the Stokes problem \eqref{eq:stokes} with the Navier boundary condition, i.e.
	\begin{equation}\label{eq:navier_bc}
	    \vec u\cdot\vec{\nu} = 0,\quad  \beta \matT \vec u + \matT\transpose\vec\sigma\del{\vec u, p}\vec\nu = \vec 0 \quad \mbox{on } \bdy\Omega.
	\end{equation}
with $\beta > 0$.
The strategy is to notice that the term $\matT\vec u$ is a compact perturbation of the pure-slip problem, and in view of \lemref{lem:lax_d} we have that the index of this new problem is zero. Therefore, its well-posedness is governed only by its finite dimensional null-space. We structure the rest of this section as follows: first, we state mild integrability assumptions on the parameter $\beta$ which still guarantee compactness of the added term; second, we show that the perturbed problem is injective by constructing a smoothing sequence as was done in \secref{s:stokes_slip_injective}; finally, we state the main result as another consequence of Corollary~\ref{cor:h_r_s}.

\begin{lemma}[$\calT_{\Omega}$ is compact]
\label{lem:navier_friction_is_compact}
Let $\calT_{\Omega}\of{\vec u}\of{\vec v} := \int_{\bdy\Omega} \beta
\matT\vec u \cdot \matT\vec v$, $s' \le r \le s$, and $\beta \in
L^q(\bdy\Omega)$ with 
\[
q \geq \frac{r}{n'} \quad\textrm{if } r > n,
\qquad
q > n-1 \quad \textrm{if } n' \le r \le n,
\qquad
q > \frac{r'}{n'} \quad\textrm{if } r < n',
\]
and $n' = n/(n-1)$.
Then the operator $\calT_\Omega : \sob1r\Omega \rightarrow \dual{\sob1{r'}\Omega}$ is compact. 
\end{lemma}
\begin{proof}
We first observe that the projection operator $\vec T = \vec I - \vec \nu \otimes \vec \nu$ is in $L^\infty(\bdy\Omega)$. We employ Sobolev embeddings and the trace theorems. We only consider the case $n' \le r \le n$ because the other two are similar. We have the following embeddings
\begin{align*}
 \cembed{\sob1r\Omega}{L^p(\bdy\Omega)} \quad 
         \frac{1}{p} > \frac{n'}{r} - \frac{1}{n-1} ,  \\
 \cembed{\sob1{r'}\Omega}{L^t(\bdy\Omega)} \quad 
         \frac{1}{t} = \frac{n'}{r'} - \frac{1}{n-1} ,
\end{align*}
the former being compact. Since 
\[
  \Big( \frac{n'}{r} - \frac{1}{n-1} \Big) 
  + \Big( \frac{n'}{r'} - \frac{1}{n-1} \Big) 
  = 1 - \frac{1}{n-1} ,
\]
for any $q>n-1$ there exists a $p$ satisfying the above inequality as
well as $\frac{1}{q} + \frac{1}{p} + \frac{1}{t} = 1$. Therefore 
\[
  \abs{\calT_\Omega\of{\vec u}\of{\vec v}}
  \le C \norm{\beta}_{L^q(\bdy\Omega)} 
      \normSZ{\vec u}1r{\Omega}
      \normSZ{\vec v}1{r'}\Omega
\]
and $\calT_\Omega$ is compact.
\end{proof}

We remark that a sufficient condition on $q$ to ensure the conclusion of \lemref{lem:navier_friction_is_compact} is $q > n-1$. For simplicity from hereon we will work under this assumption on $q$.

\begin{definition}
[smoothing sequence] 
\label{defn:navier_slip_sequence}
Let $s > n$, $q > n -1$ and $t_{0} = r = s'$. The smoothing sequence conformal to $\calS_\Omega + \calT_{\Omega}$ is given by 
    \begin{align*}
        \frac 1 {t_m} &:= \frac 1 {t_{m-1}} - \frac{1}{n}\del{1 - \del{n-1}\frac{1}q}
        \quad \mbox{ for } m=1,\ldots, M
    \end{align*}
where $M \geq n \del{1-\frac{n-1}{q}}^{-1} \ceil{\frac{1}{t_0} -\frac{1}{2}}$ guarantees that $t_M \geq 2$. 
\end{definition}

\begin{lemma}
[$\calS_\Omega + \calT_{\Omega}$ is injective]
\label{lem:navier_friction_injective}
Let $s' \leq r \leq s$ and suppose $\beta$ is strictly positive in a set $\Gamma \subset \bdy\Omega$ of positive measure. If $\of{\vec u, p} \in X_{r}(\Omega)$ is a homogeneous solution to the Stokes problem with Navier slip boundary conditions \eqref{eq:navier_bc}, then $\of{\vec u, p} = \of{\vec 0, 0}$.
\end{lemma}
\begin{proof} 
Since $\Omega$ is bounded, the case $2 < r \leq s$ follows from the
embedding $\cembed{X_{r}\of\Omega}{X_{2}\of\Omega}$. The Hilbert space
case $r=2$ follows from the coercivity estimate
    \begin{align*}
        0 & = \calS\of{\vec u, p}\of{\vec u, p}  
            + \calT_{\bdy\Omega}\of{\vec u}\of{\vec u} \\
          & = \eta \normL{\gradS{\vec u}}2{\Omega}^2 
            + \int_{\bdy\Omega}\beta \abs{\matT\vec u}^2 
          \geq \eta \normL{\gradS{\vec u}}2{\Omega}^2 
            + \beta_0 \normL{\matT \vec u}2{\Gamma}^2 
    \end{align*}
where $\beta_0 > 0$.     
Since $\normL{\gradS{u}}2{\Omega} = 0$ we recall Proposition~\ref{prop:korn_first_inequality} to conclude that $\vec u$ is an element of $Z(\Omega)$, i.e. it is an affine vector field of the form $\vec u\of{\vec x}= \mat A \vec x + \vec b$ with $\vec u \cdot \vnu = 0$ on $\bdy\Omega$. 
By using that $\matT \vec u = 0$ a.e. on $\Gamma$, we conclude that $\vec u = \vec 0$. The uniqueness of $p$, up to a constant, follows as in \secref{s:hilbert_case}.

To obtain injectivity for $s' \leq r < 2$ we suppose that $\of{\vec u, p} \in X_{t_{0}}\of{\Omega}$ is a homogeneous solution to $\calS_\Omega + \calT_{\Omega}$ with $t_0 = r = s'$. We then use the smoothing property of the Stokes operator, and the same induction argument as in \propref{s:sslip_somegar_injective}, to obtain in $M$ steps that  
    \[
        \of{\vec u,p}  \in X_{t_0}\of{\Omega} \cap \ldots \cap  X_{t_M}\of{\Omega} \subset X_{2}\of{\Omega}.
    \]
where $t_m$ is the sequence from Definition~\ref{defn:navier_slip_sequence}. This implies $\of{\vec u, p} = \of{\vec 0, 0}$ as desired.
\end{proof}


\begin{remark}[uniqueness] \rm
  If the set $\Gamma = \bdy\Omega$ in Lemma \ref{lem:navier_friction_injective},
  then we may take the set $Z=\emptyset$, i.e. the velocity field $\vec u$ is unique and the pressure is unique up to a constant, \cite{HBeiraodaVeiga_2004}.
\end{remark}

\begin{theorem}[slip with friction]
Let $\Omega$ be a bounded domain of class $W^{2-1/s}_{s}$ and $\beta$ satisfy the assumptions of Lemmas \ref{lem:navier_friction_is_compact} and \ref{lem:navier_friction_injective}. 
For every $\calF \in \dual {X_{r'}\of{\Omega}}$ there exists a unique $\of{\vec u, p} \in X_r\of{\Omega}$ such that 
    \begin{align*}
        \calS_\Omega\of{\vec u, p}\of{\vec v, q} + \calT_{\Omega}\of{\vec u}\of{\vec v}
            &= \calF\of{\vec v, q}  \quad \forall \of{\vec v, q} \in X_{r'}\of{\Omega}, 
     \intertext{and}
        \norm{(\vec u,p)}_{X_r\of{\Omega}} 
            &\leq  C_{\Omega, \eta, n, r}\norm{\calF}_{\dual{X_{r'}\of{\Omega}}},
    \end{align*}
where $s'\le r \le s$.
\end{theorem}
\begin{proof}
The proof relies on the boundedness of $\Omega$ and the compactness of $\calT_{\Omega}$. We start by noting that $\calS_\Omega^{-1}$ is a pseudo-inverse of $\calS_\Omega + \calT_{\Omega}$, i.e. 
    \begin{equation*}
        \calS_\Omega^{-1}\of{\calS_\Omega + \calT_{\Omega}} 
            = \calI_{X_r} + \calS_{\Omega}^{-1}\calT_{\Omega} 
        ,\quad 
        \del{\calS_\Omega + \calT_{\Omega}}\calS_{\Omega}^{-1} 
            = \calI_{\dual X_{r'}} + \calT_{\Omega}\calS_{\Omega}^{-1},
    \end{equation*}
because $\calS_\Omega^{-1} \calT_\Omega$ and $\calT_\Omega\calS_\Omega^{-1}$ being the product of a bounded operator and a compact one are compact. Moreover, in view of \lemref{lem:lax_c} we have that 
    \begin{equation*}
        \ind\of{\calS_{\Omega} + \calT_{\Omega}} = -\ind \calS_{\Omega}^{-1} = 0.
    \end{equation*}
Using \lemref{lem:navier_friction_injective} and the definition of the index we have that $\codim R_{\calS_{\Omega} + \calT_{\Omega}} = \dim N_{\calS_{\Omega} + \calT_{\Omega}} = 0$, i.e. $\calS_{\Omega} + \calT_{\Omega}$ is bijective. The Open Mapping Theorem guarantees the asserted estimate. 
\end{proof}


{
  \bibliographystyle{plain}
  \bibliography{references}
}

\end{document}